\documentclass[a4paper,english]{amsart}
\usepackage[T1]{fontenc}
\usepackage[latin9]{inputenc}
\usepackage{amsbsy}
\usepackage{amsthm}
\usepackage{amssymb}

\makeatletter

\numberwithin{equation}{section}
  \theoremstyle{definition}
  \newtheorem{defn}{\protect\definitionname}[section]
  \theoremstyle{remark}
  \newtheorem{rem}{\protect\remarkname}[section]
  \theoremstyle{remark}
  \newtheorem*{rem*}{\protect\remarkname}
\newenvironment{lyxlist}[1]
{\begin{list}{}
{\settowidth{\labelwidth}{#1}
 \setlength{\leftmargin}{\labelwidth}
 \addtolength{\leftmargin}{\labelsep}
 }}
{\end{list}}
  \theoremstyle{plain}
  \newtheorem{prop}{\protect\propositionname}[section]
  \theoremstyle{plain}
  \newtheorem{cor}{\protect\corollaryname}[section]
  \theoremstyle{plain}
  \newtheorem{lem}{\protect\lemmaname}[section]
  \theoremstyle{definition}
  \newtheorem{example}{\protect\examplename}[section]
  \theoremstyle{plain}
  \newtheorem{thm}{\protect\theoremname}[section]

\usepackage{upgreek}
\usepackage{amssymb}
\let\originalleft\left
\let\originalright\right
\renewcommand{\left}{\mathopen{}\mathclose\bgroup\originalleft}
\renewcommand{\right}{\aftergroup\egroup\originalright}

\makeatother

\usepackage{babel}
  \providecommand{\definitionname}{Definition}
  \providecommand{\examplename}{Example}
  \providecommand{\lemmaname}{Lemma}
  \providecommand{\propositionname}{Proposition}
  \providecommand{\remarkname}{Remark}
\providecommand{\corollaryname}{Corollary}
\providecommand{\theoremname}{Theorem}

\begin{document}

\title{On Group-Like Magmoids}

\author{Dan Jonsson}

\address{Dan Jonsson, Department of Sociology and Work Science, University
of Gothenburg, SE 405 30 Gothenburg, Sweden. }

\email{dan.jonsson@gu.se}
\begin{abstract}
A magmoid is a non-empty set with a partial binary operation; group-like
magmoids generalize group-like magmas such as groups, monoids and
semigroups. In this article, we first consider the many ways in which
the notions of associative multiplication, identities and inverses
can be generalized when the total binary operation is replaced by
a partial binary operation. Poloids, groupoids, skew-poloids, skew-groupoids,
prepoloids, pregroupoids, skew-prepoloids and skew-pregroupoids are
then defined in terms of generalized associativity, generalized identities
and generalized inverses. Some basic results about these magmoids
are proved, and connections between poloid-like and prepoloid-like
magmoids, in particular semigroups, are derived. Notably, analogues
of the Ehresmann-Schein-Nampooribad theorem are proved.
\end{abstract}

\maketitle

\section{Introduction}

A binary operation $\mathfrak{m}$ on a set $S$ is usually defined
as a mapping that assigns some $\mathfrak{m}\left(x,y\right)\in S$
to every pair $\left(x,y\right)\in S\times S$. Algebraists have been
somewhat reluctant to work with partial functions, in particular partial
binary operations, where $\mathfrak{m}\left(x,y\right)$ is defined
only for all $\left(x,y\right)$ in some subset of $S\times S$, often
called the domain of definition of $\mathfrak{m}$. For example, in
the early 1950s Wagner \cite{key-17} pointed out that if an empty
partial transformation is interpreted as the empty relation $\emptyset$
then composition of partial transformations can be regarded as being
defined for every pair of transformations. Specifically, a partially
defined binary system of non-empty partial transformations can be
reduced to a semigroup which may contain an empty partial transformation
by interpreting composition of partial transformations as composition
of binary relations. This observation may have contributed to the
acceptance of the notion of a binary system of partial transformations
\cite{key-8}. On the other hand, the reformulation alleviated the
need to come to terms with partial binary operations as such, thus
possibly leaving significant research questions unnoticed and unanswered.

One kind of reason why algebraists have hesitated to embrace partial
operations has to do with the logical status of expressions such as
$\mathfrak{f}\left(x\right)=y$ and $\mathfrak{m}\left(x,y\right)=z$
when $\mathfrak{f}$ and $\mathfrak{m}$ are partial functions. As
Burmeister \cite{key-1} explains,
\begin{quote}
{\small{}{[}a{]} first order language for algebraic systems is usually
based on an appropriate notion of equations. Such a notion {[}...{]}
has already been around for quite a long time, but approaches as by
Kleene {[}...{]}, Ebbinghaus {[}...{]}, Markwald {[}...{]} and others
(cf. also {[}Schein{]}) used then a three valued logic for the whole
language (which might have deterred algebraists from using it). (pp.
306\textendash 7).}{\small \par}
\end{quote}
One might argue that if $\left(x,y\right)$ does not belong to the
domain of definition of $\mathfrak{m}$ then the assertion $\mathfrak{m}\left(x,y\right)=z$
is meaningless, neither true nor false. This would seem to imply that
we need some three-valued logic, where an assertion is not necessarily
either true or false, similar to Kleene's three-valued logic \cite{key-10}.
This approach is very problematic, however; one need only contemplate
the meaning of an implication containing an assertion assumed to be
neither true nor false to appreciate the complications that the use
of a three-valued logic would entail. 

\pagebreak{}

Fortunately, two-valued logic suffices to handle partial functions,
including partial binary operations. In particular, three-valued logic
is not needed if expressions of the form $\mathfrak{f}\left(x_{1},\ldots,x_{n}\right)=y$
are used only if $\left(x_{1},\ldots,x_{n}\right)$ belongs to the
domain of definition of $\mathfrak{f}$. Burmeister \cite{key-1}
 elaborated a formal logic with partial functions based on this idea.
Another, simpler way to stay within two-valued logic when dealing
with partial functions is sketched in Section 2 below. 

But there is also another important kind of reason why algebraists
have shunned partial binary operations. Ljapin and Evseev \cite{key-13}
note that 
\begin{quote}
{\small{}{[}it{]} often turns out that an idea embodying one clearly
defined concept in the theory of total operations corresponds to several
mutually inequivalent notions in the theory of partial operations,
each one reflecting one or another aspect of the idea. (p. 17).}{\small \par}
\end{quote}
For example, in a group-like system with a partial binary operation,
multiplication of elements can be associative in different ways, there
are several types of identities and inverses, and different kinds
of subsystems and homomorphisms can be distinguished. While this complicates
algebraic theories using partial operations, it may be the case that
these are differences that make a difference. Maybe the proliferation
of notions just makes the theory richer and deeper, leading to a more
profound understanding of the simpler special cases. Whether complexity
equates richness and profoundness in this connection or not is a question
that cannot really be answered \emph{a priori}; the answer must be
based on experience from more or less successful use of partial operations
in applications such as those in this article.

In sum, there are no immediate reasons to avoid partial functions
and operations. Particularly in view of the fact that some systems
with partially defined operations, such as categories and groupoids,
have received much attention for many years now, it would not be unreasonable
to use a general theory of partial operations as a foundation for
a general theory of total operations. This has not yet happened in
mainstream mathematics, however: the mainstream definition of an algebra
in Universal Algebra still uses total operations, not partially defined
operations. (While partial operations are not ignored, they are typically
treated as operations with special properties rather than as operations
of the most general form.) This article is a modest attempt to fill
a little of the resulting void by generalizing a theory employing
total operations to a theory using partial operations. More concretely,
we are concerned with ``magmoids'': generalizations of magmas obtained
by replacing the total binary operation by a partial operation. Specifically,
``group-like'' magmoids are considered; these generalize group-like
magmas such as semigroups (without zeros), monoids and groups.

Section 2 contains the definitions of magmoids and other fundamental
concepts, and introduces a convenient notation applicable to partial
binary operations and other partial mappings. It is also shown that
the basic concepts can be defined in a way that does not lead to any
logical difficulties. Sections 3 and 4 deal with the many ways in
which the notions of associative multiplication, identities and inverses
from group theory can be generalized when the total binary operation
is replaced by a partial binary operation; some other concepts related
to identities and inverses are considered in Section 5. In Sections
6 and 7 a taxonomy of group-like magmoids based on the distinctions
presented in Section 4 is developed. Some results that connect the
group-like magmoids in Section 6 to those in Section 7 are proved
in Section 8.

Much unconventional terminology is introduced in this article. This
is because the core of the article is a systematic classification
of some group-like magmoids, and the terminology reflects the logic
of this classification. In some cases, the present terminology overlaps
with traditional terminology, but the new terms are not meant to replace
other, commonly used names of familiar concepts. Rather, the naming
scheme used here is intended to call attention to similarities and
differences between the notions distinguished.

\section{Partial mappings and magmoids}

\subsection{\label{subsec:Partial-mappings}Partial mappings}

Let $X_{1},\ldots,X_{n}$ be non-empty sets. An \emph{n-ary relation}
on $X_{1},\ldots,X_{n}$, denoted $\mathfrak{r}:X_{1},\ldots,X_{n}$,
or just $\mathfrak{r}$ when $X_{1},\ldots,X_{n}$ need not be specified,
is a tuple 
\[
\left(\rho,X_{1},\ldots,X_{n}\right),
\]
where $\rho\subseteq X_{1}\!\times\!\ldots\!\times\!X_{n}$. The set
$\rho$, also denoted $\gamma_{\mathfrak{r}}$, is the \emph{graph}
of $\mathfrak{r}$. The \emph{empty relation} on $X_{1},\ldots,X_{n}$
is the tuple $\left(\emptyset,X_{1},\ldots,X_{n}\right)$. The set
\[
\left\{ x_{i}\mid\left(x_{1},\ldots,x_{i},\ldots,x_{n}\right)\in\rho\right\} ,
\]
a subset of $X_{i}$ denoted $\mathrm{pr}_{i}\,\rho$, is called the
\emph{i:th projection} of $\rho$. Note that if some projection of
$\rho$ is the empty set then $\rho$ itself is the empty set.

A \emph{binary (2-ary) relation $\mathfrak{r}:X,Y$ }is thus a tuple
\[
\left(\rho,X,Y\right)
\]
such that $\rho\subseteq X\times Y$. For $\mathfrak{r}:X,Y$ we have
$\mathrm{pr}_{1}\,\rho=\left\{ x\mid\left(x,y\right)\in\mathtt{\rho}\right\} $
and $\mathrm{pr}_{2}\,\rho=\left\{ y\mid\left(x,y\right)\in\rho\right\} $.
We call $\mathrm{pr}_{1}\,\rho$ the \emph{effective domain} of $\mathfrak{r}$,
denoted $\mathrm{edom}\;\mathtt{\mathfrak{r}}$, and $\mathrm{pr}_{2}\,\rho$
the \emph{effective codomain} of $\mathfrak{r}$, denoted $\mathrm{ecod}\;\mathfrak{r}$.
We also call $X$ the \emph{total domain} of $\mathtt{\mathfrak{r}}$,
denoted $\mathrm{tdom}\;\mathtt{\mathfrak{r}}$, and $Y$ the \emph{total
codomain} of $\mathtt{\mathfrak{r}}$, denoted $\mathrm{tcod}\;\mathfrak{r}$.
A \emph{total} binary relation is a binary relation $\mathfrak{r}$
such that $\mathrm{edom}\;\mathtt{\mathfrak{r}}=\mathrm{tdom}\;\mathtt{\mathfrak{r}}$,
while a \emph{cototal} binary relation is a binary relation $\mathfrak{r}$
such that $\mathrm{ecod}\;\mathfrak{r}=\mathrm{tcod}\;\mathfrak{r}$.
\begin{defn}
\label{def1}A \emph{functional} \emph{relation, }or\emph{ (partial)
mapping,} $\mathfrak{f}:X\rightarrow Y$ is a binary relation 
\[
\left(\phi,X,Y\right)
\]
 such that for each $x\in\mathrm{pr}_{1}\,\phi$ there is exactly
one $y\in\mathrm{pr}_{2}\,\phi$ such that $\left(x,y\right)\in\phi$.
A \emph{self-mapping} on $X$ is a mapping $\mathfrak{f}:X\rightarrow X$.
\end{defn}
We let $\mathfrak{f}\left(x\right)=y$ express the fact that $\left(x,y\right)\in\gamma_{\mathfrak{f}}$.
Consistent with this, for any $x\in\mathrm{edom}\;\mathfrak{f}$,
$\mathfrak{f}\left(x\right)$ denotes the unique element of $Y$ such
that $\left(x,\mathfrak{f}\left(x\right)\right)\in\gamma_{\mathfrak{f}}$.

Let $\mathfrak{f}$ be a self-mapping on $X$. Then $\mathfrak{f}\left(x\right)$
denotes some $x\in X$ if and only if $x\in\mathrm{edom}\:\mathfrak{f}$;
$\mathfrak{f}\left(\mathfrak{f}\left(x\right)\right)$ denotes some
$x\in X$ if and only if $x,\mathsf{\mathfrak{f}}\left(x\right)\in\mathrm{edom}\:\mathfrak{f}$;
etc. We describe such situations by saying that $\mathfrak{f}\left(x\right)$,
$\mathfrak{f}\left(\mathfrak{f}\left(x\right)\right)$, etc. are \emph{defined}. 

We let $\left(\mathfrak{f}\left(x\right)\right)$, $\left(\mathfrak{f}\left(\mathfrak{f}\left(x\right)\right)\right)$
etc. express the fact that $\mathfrak{f}\left(x\right)$, $\mathfrak{f}\left(\mathfrak{f}\left(x\right)\right)$
etc. is defined.\footnote{One could extend the scope of this notation, letting $\left(x\right),\left(x'\right),\ldots$
mean that $x,x',\ldots$ belong to $X$, but in this article I will
adhere to the more familiar, light-weight notation when dealing with
'naked' variables, writing $\left(\phi\left(x\right)\right)=y$ rather
than $\left(\phi\left(x\right)\right)=\left(x'\right)$ or $\left(\phi\left(x\right)\right)=\left(y\right)$,
etc.} We also use this notation embedded in expressions, letting $\left(\mathfrak{f}\left(x\right)\right)=y$
mean that $\mathfrak{f}\left(x\right)$ is defined and $\mathfrak{f}\left(x\right)=y$,
letting $\left(\mathfrak{f}\left(x\right)\right)=\left(\mathfrak{g}\left(x\right)\right)$
mean that $\mathfrak{f}\left(x\right)$ and $\mathfrak{g}\left(x\right)$
are defined and $\mathfrak{f}\left(x\right)=\mathfrak{g}\left(x\right)$,
etc.

It is important to note that if $\left(x,y\right)\in X\times Y$ but
$x\notin\mathrm{edom}\:\mathfrak{f}=\mathrm{pr}_{1}\,\mathfrak{\gamma_{\mathfrak{f}}}$
then $\left(x,y\right)\notin\gamma_{\mathfrak{f}}$, so if $\mathfrak{f}\left(x\right)$
is not defined then $\mathfrak{f}\left(x\right)=y$ is simply false,
not meaningless. Also, $\mathfrak{f}\left(x\right)=\mathfrak{g}\left(x\right)$
is equivalent to the condition that there is some $y\in\mathrm{ecod}\,\mathfrak{f}\cap\mathrm{ecod}\,\mathfrak{g}$
such that $\mathfrak{f}\left(x\right)=y$ and $\mathfrak{g}\left(x\right)=y$,
so such expressions do not present any new logical difficulties, although
expressions such as $\left(\mathfrak{f}\left(x\right)\right)=\left(\mathfrak{g}\left(x\right)\right)$
generally describe situations of more interest.

\subsection{\label{subsec:Binary-operations-and}Binary operations and magmoids}
\begin{defn}
\label{def2}A \emph{(partial) binary operation} on a non-empty set
$X$ is a non-empty (partial) mapping
\[
\mathfrak{m}:X\times X\rightarrow X,\qquad\left(x,y\right)\mapsto\mathfrak{m}\left(x,y\right)=:xy.
\]
A \emph{total} binary operation on $X$ is a total mapping $\mathfrak{m}:X\times X\rightarrow X$. 

A \emph{magmoid} is a non-empty set $X$ equipped with a (partial)
binary operation on $X$; a \emph{total} magmoid, or \emph{magma,}
is a non-empty set $X$ equipped with a total binary operation on
$X$.
\end{defn}
Recall that in Definition \ref{def1}, we identified a (partial) mapping
$\mathfrak{f}:X\rightarrow Y$ with a binary relation $\left(\phi,X,Y\right)$
such that $\left(x,y\right),\left(x,y'\right)\in\phi$ implies $y=y'$.
We can similarly identify a (partial) binary operation $\mathfrak{m}:X\times X\rightarrow X$
with a ternary relation 
\[
\left(\mu,X,X,X\right)
\]
such that $\left(x,y,z\right),\left(x,y,z'\right)\in\mu$ implies
$z=z'$, letting $\mathfrak{m}\left(x,y\right)=z$ mean that $\left(x,y,z\right)\in\mu$.
In this case, we have $\mathrm{edom}\:\mathfrak{m}=\left\{ \left(x,y\right)\mid\left(x,y,z\right)\in\mu\right\} $,
$\mathrm{tdom}\:\mathfrak{m}=X\times X$, $\mathrm{ecod}\:\mathfrak{m}=\left\{ z\mid\left(x,y,z\right)\in\mu\right\} $
and $\mathrm{tcod}\:\mathfrak{m}=X$.

The notion of being defined for expressions involving a self-mapping
can be extended in a natural way to expressions involving a binary
operation. We say that $xy$ is defined if and only if $\left(x,y\right)\in\mathrm{edom}\:\mathfrak{m}$;
that $\left(xy\right)z$ is defined if and only if $\left(x,y\right),\left(xy,z\right)\in\mathrm{edom}\:\mathfrak{m}$;
that $z\left(xy\right)$ is defined if and only if $\left(x,y\right),\left(z,xy\right)\in\mathrm{edom}\:\mathfrak{m}$;
and so on. Thus, if $\left(xy\right)z$ or $z\left(xy\right)$ is
defined then $xy$ is defined. In analogy with the notation $\left(\mathfrak{f}\left(x\right)\right)$,
we let $\left(xy\right)$ mean that $xy$ is defined, $\left(\left(xy\right)z\right)$
mean that $\left(xy\right)z$ is defined, $\left(x\left(yz\right)\right)$
mean that $x\left(yz\right)$ is defined, etc. Note that there is
no conflict between the usual function of parentheses, namely to specify
priority of operations, and their additional use here to show that
a function is defined for a certain argument.

It is clear that $\left(x,y\right)\notin\mathrm{edom}\:\mathfrak{m}$
implies $\mathfrak{m}\left(x,y\right)\neq z$ for every $z\in X$,
so in this case, too, we do not have to deal with logical anomalies.
\begin{rem}
We have implicitly used lazy evaluation of conjunctions in this section.
That is, the conjunction $p\wedge q$ is evaluated step-by-step according
to the following algorithm:
\[
\begin{array}{lc}
\mathrm{if}\;p\;\mathrm{is\:false\:then} & p\wedge q\;\mathrm{is}\:\mathrm{false;}\\
\mathrm{else\;if}\;q\;\mathrm{is\:false\:then} & p\wedge q\;\mathrm{is}\:\mathrm{false;}\\
\mathrm{else} & p\wedge q\;\mathrm{is}\;\mathrm{true.}
\end{array}
\]
For example, if $x\notin\mathrm{edom}\:\mathfrak{f}$ then $\mathfrak{f}\left(\mathfrak{f}\left(x\right)\right)$
is not defined; if $x\in\mathrm{edom}\:\mathfrak{f}$ but instead
$\mathsf{\mathfrak{f}}\left(x\right)\notin\mathrm{edom}\:\mathfrak{f}$
then $\mathfrak{f}\left(\mathfrak{f}\left(x\right)\right)$ is also
not defined; otherwise, $\mathfrak{f}\left(\mathfrak{f}\left(x\right)\right)$
is defined. Thus, the question if $\mathfrak{f}\left(\mathfrak{f}\left(x\right)\right)$
is defined does not arise before we know if $\mathfrak{f}\left(x\right)$
is defined.
\end{rem}
\begin{rem*}
Burmeister \cite{key-1} added some new primitives to standard logic
to handle partially defined functions, and similarly the present approach
ultimately requires a slight modification of standard logic, namely
in the interpretation of conjunctions. However, it is important to
note that the dynamic (lazy) evaluation interpretation of conjunctions
is fully consistent with the static standard interpretation of conjunctions
in terms of truth tables.
\end{rem*}

\section{Conditions used in basic definitions}

A group is a magma where multiplication is associative, and where
there is an identity element and an inverse for every element. In
this section, we distinguish components of these three properties
of groups as they apply to magmoids. 

\subsection{\label{subsec:Associativity}Associativity}

In a magma $M$, an associative binary operation is one that satisfies
the condition $x\left(yz\right)=\left(xy\right)z$ for all $x,y,z\in M$.
If the magma is regarded as a magmoid $P$, we write this as
\begin{lyxlist}{00.00.0000}
\item [{$\quad$(TA)}] $\left(x\left(yz\right)\right)=\left(\left(xy\right)z\right)$
for all $x,y,z\in P$.
\end{lyxlist}
In a magmoid we can in addition define conditions that generalize
(TA):
\begin{lyxlist}{00.00.0000}
\item [{$\quad$(A1)}] If $\left(x\left(yz\right)\right)$ then $\left(x\left(yz\right)\right)=\left(\left(xy\right)z\right)$
for all $x,y,z\in P$.
\item [{$\quad$(A2)}] If $\left(\left(xy\right)z\right)$ then $\left(x\left(yz\right)\right)=\left(\left(xy\right)z\right)$
for all $x,y,z\in P$.
\item [{$\quad$(A3)}] If $\left(xy\right)$ and $\left(yz\right)$ then
$\left(x\left(yz\right)\right)=\left(\left(xy\right)z\right)$ for
all $x,y,z\in P$.
\end{lyxlist}
These elementary conditions concern different aspects of associativity,
and can be used as building blocks for constructing more complex conditions;
see Section \ref{subsec:Magmoids-according-to}. 

In semigroups, we can omit parentheses, writing $x\left(yz\right)$
and $\left(xy\right)z$ as $xyz$ without ambiguity. In fact, it can
be shown by induction using $x\left(yz\right)=\left(xy\right)z$ that
we can write $x_{1}\cdots x_{n}$ for any $n$ without ambiguity;
this is the so-called law of general associativity. Similarly, in
a magmoid where (A1) and (A2) hold we can write $\left(x_{1}\cdots x_{n}\right)$
without ambiguity when all subproducts are defined, so we have a general
associativity law in this case, too.

Specifically, let $\pi_{x_{1},\ldots,x_{n}}$ denote any parenthesized
products of $x_{1},\ldots,x_{n}$, in this order, for example, $\left(x_{1}\left(\cdots\right)\right)$
or $\left(\left(\cdots\right)x_{n}\right)$. If (A1) and (A2) hold
and $\pi_{x_{1},\ldots,x_{n}}'$ is a parenthesized product of the
same kind then it can be shown by induction that
\[
\pi_{x_{1},\ldots,x_{n}}=\pi_{x_{1},\ldots,x_{n}}'.
\]
For example, if $\left(x\left(y\left(zu\right)\right)\right)$ then
\[
\left(x\left(y\left(zu\right)\right)\right)=\left(x\left(\left(yz\right)u\right)\right)=\left(\left(x\left(yz\right)\right)u\right)=\left(\left(\left(xy\right)z\right)u\right)=\left(\left(xy\right)\left(zu\right)\right).
\]
In other words, $\pi_{x_{1},\ldots,x_{n}}$ is uniquely determined
by the sequence $x_{1},\ldots,x_{n}$, so we can write any $\pi_{x_{1},\ldots,x_{n}}$
as $\left(x_{1}\cdots x_{n}\right)$. For example, we can write $\left(\left(xy\right)\left(zu\right)\right)$
as $\left(xyzu\right)$ without loss of information.

(A3) thus implies that if $\left(xy\right)$ and $\left(yz\right)$
then $\left(x\left(yz\right)\right)=\left(\left(xy\right)z\right)=\left(xyz\right)$.
By repeatedly applying (A1) and (A2) together with (A3), we can generalize
(A3); for example, if $\left(xyz\right)$ and $\left(zuv\right)$
then $\left(xyzuv\right)$. We can also generalize (A3) by applying
it repeatedly; for example, if $\left(xy\right)$, $\left(yz\right)$
and $\left(zu\right)$ then $\left(xyzu\right)$. Combining these
two ways of generalizing (A3), it becomes possible to make inferences
such as ''$\mathrm{if}\;\left(x_{1}x_{2}\right)\;\mathrm{and}\;\left(x_{2}x_{3}x_{4}x_{5}\right)\;\mathrm{and}\;\left(x_{5}x_{6}x_{7}\right)\mathrm{\ then}\;\left(x_{1}x_{2}x_{3}x_{4}x_{5}x_{6}x_{7}\right)$''.

Note that we can retain certain redundant inner parentheses for emphasis.
For example, if $\left(xx^{-1}x\right)=x$ then we can write $\left(\left(xx^{-1}x\right)y\right)=\left(xy\right)$
instead of\linebreak{}
 $\left(xx^{-1}xy\right)=\left(xy\right)$ to clarify how the equality
is established.

\subsection{\label{subsec:Identities}Identities}

In a magma $M$, an identity is an element $e\in M$ such that $ex=xe=x$
for all $x\in M$. More generally, a left (resp. right) identity is
an element $e\in M$ such that $ex=x$ (resp. $xe=x$) for all $x\in M$.
In a magmoid $P$, the conditions defining left and right identities
take the following forms:
\begin{lyxlist}{00.00.0000}
\item [{$\quad$(TU1)}] If $x\in P$ then $\left(ex\right)=x$.
\item [{$\quad$(TU2)}] If $x\in P$ then $\left(xe\right)=x$.
\end{lyxlist}
In a magmoid we can in addition define conditions generalizing (TU1)
and (TU2):
\begin{lyxlist}{00.00.0000}
\item [{$\quad$(GU1)}] If $x\in P$ and $\left(ex\right)$ then $\left(ex\right)=x$.
\item [{$\quad$(GU2)}] If $x\in P$ and $\left(xe\right)$ then $\left(xe\right)=x$.
\item [{$\quad$(LU1)}] There is some $x\in P$ such that $\left(ex\right)=x$.
\item [{$\quad$(LU2)}] There is some $x\in P$ such that $\left(xe\right)=x$.
\end{lyxlist}
In Section \ref{subsec:Types-of-units}, we define different types
of identities (units) in magmoids by means of combinations of these
conditions.

\subsection{\label{subsec:Inverses}Inverses}

In a magma $M$, an inverse of $x\in M$ is an element $x^{-1}\in M$
such that $xx^{-1}=x^{-1}x=e$ for some identity $e\in M$, while
a right (resp. left) inverse of $x\in M$ is an element $x^{-1}\in M$
such that $xx^{-1}=e$ (resp. $x^{-1}x=e$) for some identity $e\in M$.
We can also define a\emph{ right (resp. left) semi-inverse} of $x\in M$
as an element $x^{-1}\in M$ such that $xx^{-1}=e$ (resp. $x^{-1}x=e$)
for some left or right identity $e\in M$; these are the most fundamental
notions. In a magmoid $P$, the conditions defining a left or right
semi-inverse $x^{-1}$ take the following forms:
\begin{lyxlist}{00.00.0000}
\item [{$\quad$(TI1)}] If $x\in P$ then $\left(xx^{-1}\right)=e$ for
some $e$ satisfying (TU1). 
\item [{$\quad$(TI2)}] If $x\in P$ then $\left(x^{-1}x\right)=e$ for
some $e$ satisfying (TU1). 
\item [{$\quad$(TI3)}] If $x\in P$ then $\left(xx^{-1}\right)=e$ for
some $e$ satisfying (TU2).
\item [{$\quad$(TI4)}] If $x\in P$ then $\left(x^{-1}x\right)=e$ for
some $e$ satisfying (TU2). 
\end{lyxlist}
In addition, we can define conditions that generalize (TI1) \textendash{}
(TI4):
\begin{lyxlist}{00.00.0000}
\item [{$\quad$(GI1)}] If $x\in P$ then $\left(xx^{-1}\right)=e$ for
some $e$ satisfying (GU1). 
\item [{$\quad$(GI2)}] If $x\in P$ then $\left(x^{-1}x\right)=e$ for
some $e$ satisfying (GU1). 
\item [{$\quad$(GI3)}] If $x\in P$ then $\left(xx^{-1}\right)=e$ for
some $e$ satisfying (GU2).
\item [{$\quad$(GI4)}] If $x\in P$ then $\left(x^{-1}x\right)=e$ for
some $e$ satisfying (GU2). 
\item [{$\quad$(LI1)}] If $x\in P$ then $\left(xx^{-1}\right)=e$ for
some $e$ satisfying (LU1). 
\item [{$\quad$(LI2)}] If $x\in P$ then $\left(x^{-1}x\right)=e$ for
some $e$ satisfying (LU1). 
\item [{$\quad$(LI3)}] If $x\in P$ then $\left(xx^{-1}\right)=e$ for
some $e$ satisfying (LU2).
\item [{$\quad$(LI4)}] If $x\in P$ then $\left(x^{-1}x\right)=e$ for
some $e$ satisfying (LU2). 
\end{lyxlist}
In Section \ref{subsec:Types-of-inverses;}, different types of inverses
in magmoids will be defined in terms of conditions of this kind.

\section{Basic definitions}

Some of the possible combinations of elementary conditions in Section
3 will be used in this section to define different types of associatitvity,
units (identities) and inverses. There are two main themes in this
section: one-sided versus two-sided notions, and local versus global
notions.

\subsection{\label{subsec:Magmoids-according-to}Magmoids according to types
of associativity}
\begin{defn}
\label{def3}Let $P$ be a magmoid, $x,y,z\in P$. Consider the following
con\-ditions:
\begin{lyxlist}{00.00.0000}
\item [{(S1)}] If $\left(x\left(yz\right)\right)$ or if $\left(xy\right)$
and $\left(yz\right)$ then $\left(x\left(yz\right)\right)=\left(\left(xy\right)z\right)$.
\item [{(S2)}] If $\left(\left(xy\right)z\right)$ or if $\left(xy\right)$
and $\left(yz\right)$ then $\left(x\left(yz\right)\right)=\left(\left(xy\right)z\right)$.
\item [{(S3)}] If $\left(x\left(yz\right)\right)$ or $\left(\left(xy\right)z\right)$
or if $\left(xy\right)$ and $\left(yz\right)$ then $\left(x\left(yz\right)\right)=\left(\left(xy\right)z\right)$.
\end{lyxlist}
A \emph{left semigroupoid} is a magmoid satisfying (S1), a \emph{right
semigroupoid} is a magmoid satisfying (S2), and a \emph{(two-sided)
semigroupoid} is a magmoid satisfying (S3). 
\end{defn}
It is clear that if $P$ is a magma and at least one of the conditions
(S1) \textendash{} (S3) is satisfied then $P$ is a semigroup; conversely,
if $P$ is a semigroup then $P$ is a magma where (S1) \textendash{}
(S3) are trivial implications so that they are all satisfied.

In view of the symmetry between left and right semigroupoids essentially
only two cases will be considered below, namely left (or right) semigroupoids
and (two-sided) semigroupoids. For later use, we note that a left
semigroupoid is a magmoid such that for all $x,y,z\in P$
\begin{equation}
\left(x\left(yz\right)\right)\rightarrow\left(x\left(yz\right)\right)=\left(\left(xy\right)z\right),\quad\left(\left(xy\right)\wedge\left(yz\right)\right)\rightarrow\left(x\left(yz\right)\right)=\left(\left(xy\right)z\right),\label{eq:l1}
\end{equation}
 while a semigroupoid is a left semigroupoid such that in addition
to (\ref{eq:l1}) we have
\begin{equation}
\left(\left(xy\right)z\right)\rightarrow\left(x\left(yz\right)\right)=\left(\left(xy\right)z\right).\label{eq:l2}
\end{equation}

\subsection{\label{subsec:Types-of-units}Types of units in magmoids}
\begin{defn}
Let $P$ be a magmoid, $x\in P$. 
\begin{enumerate}
\item A \emph{(global) left unit} is some $\epsilon\in P$ such that if
$\left(\epsilon x\right)$ then $\left(\epsilon x\right)=x$, while
a \emph{(global) right unit} is some $\varepsilon\in P$ such that
if $\left(x\varepsilon\right)$ then $\left(x\varepsilon\right)=x$. 
\item A \emph{(global) two-sided unit} is some $e\in P$ which is a (global)
left unit and a (global) right unit.
\end{enumerate}
\end{defn}
We denote the set of left units, right units, and two-sided units
in $P$ by $\left\{ \epsilon\right\} _{\!P}$, $\left\{ \varepsilon\right\} _{\!P}$,
and $\left\{ e\right\} _{\!P}$, respectively. 
\begin{defn}
Let $P$ be a magmoid, $x\in P$.
\begin{enumerate}
\item A \emph{local left unit for} $x$ is some $\lambda_{x}\in P$ such
that $\left(\lambda_{x}x\right)=x$; a \emph{local right unit for}
$x$ is some $\rho_{x}\in P$ such that $x=\left(x\rho_{x}\right)$. 
\item A\emph{ twisted left unit }$\varphi_{x}\in P$\emph{ for} $x$ is
a local left unit for $x$ which is a right unit, while a\emph{ twisted
right unit }$\psi_{x}\in P$\emph{ for} $x$ is a local right unit
for $x$ which is a left unit.
\item A \emph{left effective unit }$\ell_{x}\in P$ \emph{for} $x$ is a
local left unit for $x$ which is a two-sided unit, while a \emph{right
effective unit }$r_{x}\in P$ \emph{for} $x$ is a local right unit
for $x$ which is a two-sided unit.
\end{enumerate}
\end{defn}
We denote the set of local left (resp. local right) units for $x\in P$
by $\left\{ \lambda\right\} _{x}$ (resp. $\left\{ \rho\right\} _{x}$),
the set of twisted left (resp. twisted right) units for $x\in P$
by $\left\{ \varphi\right\} _{x}$ (resp. $\left\{ \psi\right\} _{x}$),
and the set of left effective (resp. right effective) units for $x\in P$
by $\left\{ \ell\right\} _{x}$ (resp. $\left\{ r\right\} _{x}$).
We also set $\left\{ \lambda\right\} _{\!P}=\cup_{x\in P}\left\{ \lambda\right\} _{x}$,
$\left\{ \rho\right\} _{\!P}=\cup_{x\in P}\left\{ \rho\right\} _{x}$,
$\left\{ \varphi\right\} _{\!P}=\cup_{x\in P}\left\{ \varphi\right\} _{x}$,
$\left\{ \psi\right\} _{\!P}=\cup_{x\in P}\left\{ \psi\right\} _{x}$,
$\left\{ \ell\right\} _{\!P}=\cup_{x\in P}\left\{ \ell\right\} _{x}$
and $\left\{ r\right\} _{\!P}=\cup_{x\in P}\left\{ r\right\} _{x}$. 

\subsection{\label{subsec:Types-of-inverses;}Types of inverses in magmoids}
\begin{defn}
Let $P$ be a magmoid, $x\in P$. 
\begin{enumerate}
\item A \emph{pseudoinverse} of $x$ is some $x^{\left(-1\right)}\in P$
such that $\left(xx^{\left(-1\right)}\right)\in\left\{ \lambda\right\} _{x}$
and $\left(x^{\left(-1\right)}x\right)\in\left\{ \rho\right\} _{x}$. 
\item A \emph{right preinverse} of $x$ is some $x^{-1}\in P$ such that
$\left(xx^{-1}\right)\in\left\{ \lambda\right\} _{x}$ and $\left(x^{-1}x\right)\in\left\{ \lambda\right\} _{x^{-1}}$. 
\item A \emph{left preinverse} of $x$ is some $x^{-1}\in P$ such that
$\left(x^{-1}x\right)\in\left\{ \rho\right\} {}_{x}$ and $\left(xx^{-1}\right)\in\left\{ \rho\right\} {}_{x^{-1}}$. 
\item A \emph{(two-sided) preinverse} of $x$ is some $x^{-1}\in P$ such
that we have\linebreak{}
 $\left(xx^{-1}\right)\in\left\{ \lambda\right\} _{x}\cap\left\{ \rho\right\} _{x^{-1}}$
and $\left(x^{-1}x\right)\in\left\{ \rho\right\} _{x}\cap\left\{ \lambda\right\} _{x^{-1}}$.
\end{enumerate}
\end{defn}
Thus, $x^{\left(-1\right)}$ is a pseudoinverse of $x$ if and only
if $\left(\left(xx^{\left(-1\right)}\right)x\right)=\left(x\left(x^{\left(-1\right)}x\right)\right)=x$,
$x^{-1}$ is a right preinverse of $x$ if and only if $\left(\left(xx^{-1}\right)x\right)=x$
and $\left(\left(x^{-1}x\right)x^{-1}\right)=x^{-1}$, $x^{-1}$ is
a left preinverse of $x$ if and only if $\left(x\left(x^{-1}x\right)\right)=x$
and $\left(x^{-1}\left(xx^{-1}\right)\right)=x^{-1}$, and $x^{-1}$
is a preinverse of $x$ if and only if $\left(\left(xx^{-1}\right)x\right)=\left(x\left(x^{-1}x\right)\right)=x$
and $\left(\left(x^{-1}x\right)x^{-1}\right)=\left(x^{-1}\left(xx^{-1}\right)\right)=x^{-1}$.

Let $\mathbf{J}$, $\mathbf{I}^{+}$, $\mathfrak{\mathbf{I}}^{*}$
and $\mathfrak{\mathbf{I}}$ be binary relations on a magmoid $P$
such that $x\,\mathbf{J}\,\overline{x}$ if and only if $\left(\left(x\overline{x}\right)x\right)=\left(x\left(\overline{x}x\right)\right)=x$,
$x\,\mathbf{I}^{+}\,\overline{x}$ if and only if $\left(\left(x\overline{x}\right)x\right)=x$
and $\left(\left(\overline{x}x\right)\overline{x}\right)=\overline{x}$,
$x\,\mathbf{I}^{*}\,\overline{x}$ if and only if $\left(x\left(\overline{x}x\right)\right)=x$
and $\left(\overline{x}\left(x\overline{x}\right)\right)=\overline{x}$,
and $x\,\mathbf{I}\,\overline{x}$ if and only if $x\,\mathfrak{\mathbf{I}^{+}}\,\overline{x}$
and $x\,\mathbf{I}^{*}\,\overline{x}$. In terms of these relations,
$\overline{x}$ is a pseudoinverse of $x$ if and only if $x\,\mathbf{J}\,\overline{x}$,
a right preinverse of $x$ if and only if $x\,\mathbf{I}^{+}\,\overline{x}$,
a left preinverse of $x$ if and only if $x\,\mathbf{I}^{*}\,\overline{x}$,
and a preinverse of $x$ if and only if $x\,\mathbf{I}\,\overline{x}$.
Note that $\mathbf{I}^{+}$, $\mathbf{I}^{*}$ and $\mathbf{I}$ are
symmetric relations.

It is useful to have some special notation for sets of pseudoinverses
and preinverses, and we set $\mathbf{J}\left\{ \right\} _{x}=\left\{ \overline{x}\mid x\,\mathbf{J}\,\overline{x}\right\} $,
$\mathbf{I}^{+}\!\left\{ \right\} _{x}=\left\{ \overline{x}\mid x\,\mathbf{I}^{+}\,\overline{x}\right\} $,
$\mathbf{I}^{*}\!\left\{ \right\} _{x}=\left\{ \overline{x}\mid x\,\mathbf{I}^{*}\,\overline{x}\right\} $
and $\mathbf{I}\left\{ \right\} _{x}=\left\{ \overline{x}\mid x\,\mathbf{I}\,\overline{x}\right\} $.
\begin{defn}
Let $P$ be a magmoid, $x\in P$. 
\begin{enumerate}
\item A \emph{strong pseudoinverse of $x$} is a pseudoinverse $x^{\left(-1\right)}$
of $x$ such that\linebreak{}
 $\left(xx^{\left(-1\right)}\right),\left(x^{\left(-1\right)}x\right)\in\left\{ e\right\} {}_{P}$.
\item A \emph{strong right preinverse of $x$} is a right preinverse $x^{-1}$
of $x$ such that $\left(xx^{-1}\right),\left(x^{-1}x\right)\in\left\{ \varepsilon\right\} {}_{P}$.
\item A \emph{strong left preinverse of $x$} is a left preinverse $x^{-1}$
of $x$ such that\linebreak{}
 $\left(x^{-1}x\right),\left(xx^{-1}\right)\in\left\{ \epsilon\right\} {}_{P}$.
\item A \emph{strong} \emph{(two-sided) preinverse of $x$} is a preinverse
$x^{-1}$ of $x$ such that $\left(xx^{-1}\right),\left(x^{-1}x\right)\in\left\{ e\right\} {}_{P}$.
\end{enumerate}
\end{defn}
By this definition, a strong right preinverse of $x$ is a right preinverse
$x^{-1}$ of $x$ such that $\left(y\left(xx^{-1}\right)\right)=\left(y\left(x^{-1}x\right)\right)=y$
for all $y\in P$, and a strong preinverse of $x$ is a preinverse
$x^{-1}$ of $x$ such that $\left(y\left(xx^{-1}\right)\right)\!=\!\left(y\left(x^{-1}x\right)\right)\!=\left(\left(xx^{-1}\right)y\right)=\left(\left(x^{-1}x\right)y\right)=y$
for all $y\in P$. Strong left preinverses and strong pseudoinverses
have corresponding properties.

We let $x\,\boldsymbol{J}\,\overline{x}$, $x\,\boldsymbol{I}^{+}\,\overline{x}$,
$x\,\boldsymbol{I}^{*}\,\overline{x}$ and $x\,\boldsymbol{I}\,\overline{x}$
mean that $\overline{x}$ is a strong pseudoinverse, strong right
preinverse, strong left preinverse and strong preinverse of $x$,
respectively. It is clear that $\boldsymbol{I}^{+}$, $\boldsymbol{I}^{*}$
and $\boldsymbol{I}$ are symmetrical relations. We also set $\boldsymbol{J}\left\{ \right\} _{x}=\left\{ \overline{x}\mid x\,\boldsymbol{J}\,\overline{x}\right\} $,
$\boldsymbol{I}^{+}\!\left\{ \right\} _{x}=\left\{ \overline{x}\mid x\,\boldsymbol{I}^{+}\,\overline{x}\right\} $,
$\boldsymbol{I}^{*}\!\left\{ \right\} _{x}=\left\{ \overline{x}\mid x\,\boldsymbol{I}^{*}\,\overline{x}\right\} $
and $\boldsymbol{I}\left\{ \right\} _{x}=\left\{ \overline{x}\mid x\,\boldsymbol{I}\,\overline{x}\right\} $.

\subsection{Canonical local units\label{subsec:Canonical}}

Recall that if $x^{-1}$ is a left, right or two-sided pre\-inverse
of $x$ then $\left(xx^{-1}\right)\in\left\{ \lambda\right\} _{x}\cup\left\{ \rho\right\} _{x^{-1}}$
and $\left(x^{-1}x\right)\in\left\{ \rho\right\} _{x}\cup\left\{ \lambda\right\} _{x^{-1}}$.
We set $\boldsymbol{\uplambda}_{x}=\boldsymbol{\uprho}_{x^{-1}}=\left(xx^{-1}\right)$,
$\boldsymbol{\uprho}_{x}=\boldsymbol{\uplambda}_{x^{-1}}=\left(x^{-1}x\right)$,
and call such local units \emph{canonical} local units. We may also
use the notation $\left\{ \boldsymbol{\uplambda}\right\} _{x}$ (resp.
$\left\{ \boldsymbol{\uprho}\right\} _{x}$) for the set of canonical
local left (resp. right) units for $x$, and the notation $\left\{ \boldsymbol{\uplambda}\right\} _{x^{-1}}$
(resp. $\left\{ \boldsymbol{\uprho}\right\} _{x^{-1}}$) for the set
of canonical local left (resp. right) units for $x^{-1}$. We also
set $\left\{ \boldsymbol{\uplambda}\right\} _{\!P}=\cup_{x\in P}\left\{ \boldsymbol{\uplambda}\right\} _{x}$
and $\left\{ \boldsymbol{\uprho}\right\} _{\!P}=\cup_{x\in P}\left\{ \boldsymbol{\uprho}\right\} _{x}$.

A canonical local unit of the form $\left(xx^{.-1}\right)$ (resp.
$\left(x^{-1}x\right)$) such that $\left(xx'\right)=\left(xx''\right)$
(resp. $\left(x'x\right)=\left(x''x\right)$) for any inverses $x',x''$
of $x$ is said to be \emph{unique}.

If we regard $x^{-1}$ not as a preinverse of $x\in P$ but just as
an element $x^{-1}\in P$, we write $\boldsymbol{\uplambda}_{\left(x^{-1}\right)}$
instead of $\boldsymbol{\uplambda}_{x^{-1}}$ and $\boldsymbol{\uprho}_{\left(x^{-1}\right)}$
instead of $\boldsymbol{\uprho}_{x^{-1}}$, setting 
\[
\boldsymbol{\uplambda}_{\left(x^{-1}\right)}=\left(x^{-1}\left(x^{-1}\right)^{-1}\right),\quad\boldsymbol{\uprho}_{\left(x^{-1}\right)}=\left(\left(x^{-1}\right)^{-1}x^{-1}\right)
\]
for some preinverse $\left(x^{-1}\right)^{-1}$ of $x^{-1}$. Note,
though, that if $x\,\mathbf{I}\,x^{-1}$ then $x^{-1}\,\mathbf{I}\,x$,
so $x\in\mathbf{I}\left\{ \right\} _{x^{-1}}$, so if $\boldsymbol{\uplambda}_{\left(x^{-1}\right)}$
is unique then $\left(x^{-1}x\right)=\left(x^{-1}\left(x^{-1}\right)^{-1}\right)$,
meaning that $\boldsymbol{\uplambda}_{x^{-1}}=\boldsymbol{\uplambda}_{\left(x^{-1}\right)}$;
similarly, if $\boldsymbol{\uprho}_{\left(x^{-1}\right)}$ is unique
then $\boldsymbol{\uprho}_{x^{-1}}=\boldsymbol{\uprho}_{\left(x^{-1}\right)}$.
These identities hold for left and right preinverses as well, since
$\mathbf{I}^{+}$ and $\mathbf{I}^{*}$ are symmetric relations.
\begin{rem}
It is natural to write $y=\lambda_{x}$ when $\left\{ y\right\} =\left\{ \lambda\right\} _{x}$,
$y=\rho_{x}$ when \linebreak{}
$\left\{ y\right\} =\left\{ \rho\right\} _{x}$, $y=x^{-1}$ when
$\left\{ y\right\} =\mathbf{I}{}^{+}\!\left\{ \right\} _{x}$, $\left\{ y\right\} =\mathbf{I}^{*}\!\left\{ \right\} _{x}$,
$\left\{ y\right\} =\mathbf{I}\left\{ \right\} _{x}$, $\left\{ y\right\} =\boldsymbol{I}{}^{+}\!\left\{ \right\} _{x}$,
$\left\{ y\right\} =\boldsymbol{I}^{*}\!\left\{ \right\} _{x}$ or
$\left\{ y\right\} =\boldsymbol{\boldsymbol{I}}\left\{ \right\} _{x}$,
and so on.
\end{rem}

\section{Idempotents and involution}

\subsection{\label{subsec:Idempotents-in-magmoids}Idempotents in magmoids}
\begin{defn}
An\emph{ idempotent} in a magmoid $P$ is some $i\in P$ such that
$\left(ii\right)=i$.
\end{defn}
\begin{prop}
\label{pro51}Let $P$ be a magmoid, $i\in P$. If $\left(ii\right)=i$
then $i\in\left\{ \lambda\right\} _{i}\cap\left\{ \rho\right\} _{i}$. 
\end{prop}
\begin{proof}
If $\left(ii\right)=i$ then $i\in\left\{ \lambda\right\} _{i}$ and
$i\in\left\{ \rho\right\} _{i}$. 
\end{proof}
\begin{cor}
\label{cor51}Let $P$ be a magmoid with unique local units, $i\in P$.
If $\left(ii\right)=i$ then $i=\lambda_{i}=\rho_{i}$.
\end{cor}
\begin{prop}
\label{pro52}Let $P$ be a magmoid, $i\in P$. If $\left(ii\right)=i$
then $i\in\mathbf{I}^{+}\!\left\{ \right\} _{i}$, $i\in\mathbf{I}^{*}\!\left\{ \right\} _{i}$
and $i\in\mathbf{I}\left\{ \right\} _{i}$. 
\end{prop}
\begin{proof}
If $\left(ii\right)=i$ then $\left(\left(ii\right)i\right)=\left(i\left(ii\right)\right)=i$,
so $i\,\mathfrak{\mathbf{I}^{+}}i$, $i\,\mathbf{I}^{*}i$ and $i\,\mathbf{I}\,i$. 
\end{proof}
\begin{cor}
\label{cor52}Let $P$ be a magmoid where right preinverses, left
preinverses or (two-sided) preinverses are unique, $i\in P$. If $\left(ii\right)=i$
then $i^{-1}=i$.
\end{cor}
\begin{proof}
We have $i\in\mathbf{I}^{+}\!\left\{ \right\} _{i}=\left\{ i^{-1}\right\} $
or $i\in\mathbf{I}^{*}\!\left\{ \right\} _{i}=\left\{ i^{-1}\right\} $
or $i\in\mathbf{I}\left\{ \right\} _{i}=\left\{ i^{-1}\right\} $.
\end{proof}
\begin{prop}
Let $P$ be a left (resp. right) semigroupoid with unique right (resp.
left) preinverses, or a semigroupoid with unique preinverses. Then
$\left(xx^{-1}\right)$ \textup{and $\left(x^{-1}x\right)$} are idempotents
for every $x\in P$, and for every idempotent $i\in P$ there is some
$x\in P$ such that $i=\left(xx^{-1}\right)=\left(x^{-1}x\right)$.
\end{prop}
\begin{proof}
It suffices two consider the first two cases. In both of these, $\left(xx^{-1}\right)$
and $\left(x^{-1}x\right)$ so that $\left(x\left(x^{-1}\left(xx^{-1}\right)\right)\right)$,
$\left(x^{-1}\left(x\left(x^{-1}x\right)\right)\right)$, $\left(\left(\left(x^{-1}x\right)x^{-1}\right)x\right)$
and\linebreak{}
 $\left(\left(\left(xx^{-1}\right)x\right)x^{-1}\right)$, and in
a left (resp. right) semigroupoid we have
\begin{gather*}
\left(x\left(x^{-1}\left(xx^{-1}\right)\right)\right)=\left(\left(xx^{-1}\right)\left(xx^{-1}\right)\right)=\left(\left(\left(xx^{-1}\right)x\right)x^{-1}\right)=\left(xx^{-1}\right),\\
\left(x^{-1}\left(x\left(x^{-1}x\right)\right)\right)=\left(\left(x^{-1}x\right)\left(x^{-1}x\right)\right)=\left(\left(\left(x^{-1}x\right)x^{-1}\right)x\right)=\left(x^{-1}x\right),\\
\mathrm{(resp.\quad}\left(\left(\left(x^{-1}x\right)x^{-1}\right)x\right)=\left(\left(x^{-1}x\right)\left(x^{-1}x\right)\right)=\left(x^{-1}\left(x\left(x^{-1}x\right)\right)\right)=\left(x^{-1}x\right),\\
\left(\left(\left(xx^{-1}\right)x\right)x^{-1}\right)=\left(\left(xx^{-1}\right)\left(xx^{-1}\right)\right)=\left(x\left(x^{-1}\left(xx^{-1}\right)\right)\right)=\left(xx^{-1}\right)\mathrm{\:).}
\end{gather*}
Conversely, if $i\in P$ is an idempotent then $i=i^{-1}$ by Corollary
\ref{cor52}, so $i=\left(ii\right)=\left(ii^{-1}\right)=\left(i^{-1}i\right)$. 
\end{proof}

\subsection{\label{subsec:Involution-magmoids}Involution magmoids}

A magma $M$ may be equipped with a total self-mapping $*:x\mapsto x^{*}$
such that $\left(x^{*}\right)^{*}=x$ and $\left(xy\right)^{*}=y^{*}x^{*}$.
A total self-mapping with these properties is an anti-endomorphism
on $M$ by the second condition, and a bijection by the first condition,\footnote{If $x_{1}^{*}=x_{2}^{*}$ then $x_{1}=\left(x_{1}^{*}\right)^{*}=\left(x_{2}^{*}\right)^{*}=x_{2}$,
and if $x\in M$ then $x=\left(x^{*}\right)^{*}\in\left(M^{*}\right)^{*}\subseteq M^{*}$
since $M^{*}\subseteq M$, so $M\subseteq M^{*}$, so $M=M^{*}$.
Hence, $x\mapsto x^{*}$ is injective and surjective.} so $*$ is an anti-automorphism, called an involution for $M$.\footnote{This is the definition in semigroup theory; in general mathematics
an involution is usually defined as a self-mapping $*$ such that
$\left(x^{*}\right)^{*}=x$.} This notion can be generalized to magmoids.
\begin{defn}
A \emph{(total)} \emph{involution magmoid} is a magmoid $P$ equipped
with a total mapping
\[
*:P\rightarrow P,\qquad x\mapsto*\left(x\right)=:x^{*}
\]
such that $\left(x^{*}\right)^{*}=x$ and if $\left(xy\right)$ then
$\left(xy\right)^{*}=\left(y^{*}x^{*}\right)$ for all $x,y\in P$.
We call the function $x\mapsto x^{*}$ an \emph{involution} and $x^{*}$
the \emph{involute} of $x$.
\end{defn}
Involutes are inverse-like elements, but while inverses are defined
in terms of units of various kinds (as elaborated in Sections \ref{subsec:Inverses}
and \ref{subsec:Types-of-inverses;}), involutes are defined without
reference to units, so involutes generalize inverses to situations
where units may not be available. Conversely, however, unit-like elements
may be defined in terms of involutes.
\begin{defn}
A \emph{unity} in a magmoid $P$ with involution $*$ is some $u\in P$
such that $u^{*}=u$.
\end{defn}
For any $x\in P$, $\left(xx^{*}\right)$ and $\left(x^{*}x\right)$
are unities since $\left(xx^{*}\right)^{*}=\left(\left(x^{*}\right)^{*}x^{*}\right)=\left(xx^{*}\right)$
and $\left(x^{*}x\right)^{*}=\left(x^{*}\left(x^{*}\right)^{*}\right)=\left(x^{*}x\right)$.
If $P$ is a semigroupoid and $\left(xx^{*}x\right)=x$ then $\left(x^{*}xx^{*}x\right)=\left(x^{*}x\right)$
and $\left(xx^{*}xx^{*}\right)=\left(xx^{*}\right)$, so that $\left(xx^{*}\right)$
and $\left(x^{*}x\right)$ are idempotents.

We can use unities to define a kind of inverses, just as we used units
to define inverses in Section \ref{subsec:Types-of-inverses;}. Let
$P$ be an involution magmoid, $x\in P$. An \emph{involution pseudoinverse}
of $x$ is some $x^{\left(+\right)}\in P$ such that
\[
\left(\left(xx^{\left(+\right)}\right)x\right)=\left(x\left(x^{\left(+\right)}x\right)\right)=x,
\]
where $\left(xx^{\left(+\right)}\right)$ and $\left(x^{\left(+\right)}x\right)$
are unities, while an \emph{involution preinverse} of $x$ is some
$x^{+}\in P$ such that
\[
\left(\left(xx^{+}\right)x\right)=\left(x\left(x^{+}x\right)\right)=x,\quad\left(\left(x^{+}x\right)x^{+}\right)=\left(x^{+}\left(xx^{+}\right)\right)=x^{+},
\]
where $\left(xx^{+}\right)$ and $\left(x^{+}x\right)$ are unities.
It is easy to show that in a semigroupoid there is at most one involution
preinverse for each $x\in P$. 

In the semigroupoid of all matrices over $\mathbb{R}$ or $\mathbb{C}$,
where $\left(MN\right)$ if and only if $M$ is a $p\times q$ matrix
and $N$ is a $q\times r$ matrix, the transpose $A^{\mathrm{T}}$
and conjugate transpose $A^{\dagger}$ of a $p\times q$ matrix $A$
are involutes of $A$.\footnote{If $\alpha_{ji}=a_{ij}$, $\beta_{kj}=b_{jk}$ and $\gamma_{ki}=\sum_{j}a_{ij}b_{jk}$
then $\gamma_{ki}=\sum_{j}\beta{}_{kj}\alpha_{ji}$ and $\overline{\gamma_{ki}}=\sum_{j}\overline{\beta_{kj}}\overline{\alpha{}_{ji}}$.} The unities are then symmetric or Hermitian matrices, that is, matrices
such that $A^{\mathrm{T}}=A$ or $A^{\dagger}=A$. The involution
preinverse of $A$ is the unique\footnote{Involution inverses are unique when they exist, and it can be shown
that every matrix $A$ has an involution inverse $A^{+}$ with respect
to the involutions $A\mapsto A^{\mathrm{T}}$ and $A\mapsto A^{\dagger}$.} $q\times p$ matrix $A^{+}$ such that $\left(AA^{+}A\right)=A$,
$\left(A^{+}AA^{+}\right)=A^{+}$ and such that the $p\times p$ matrix
$\left(AA^{+}\right)$ and the $q\times q$ matrix $\left(A^{+}A\right)$
are unities, that is, symmetric or Hermitian matrices for which $\left(AA^{+}\right)^{*}=\left(AA^{+}\right)$
and $\left(A^{+}A\right)^{*}=\left(A^{+}A\right)$. $A^{+}$ is known
as the Moore-Penrose inverse of $A$.

Total involutions can be generalized to partial involutions; this
notion is also of interest. For example, it is easy to verify that
the partial function $A\mapsto A^{-1}$, which associates every invertible
matrix with its inverse, is a partial involution in the semigroupoid
of matrices over $\mathbb{R}$ or $\mathbb{C}$. In every subsemigroupoid
of invertible $n\times n$ matrices over $\mathbb{R}$ or $\mathbb{C}$,
$A\mapsto A^{-1}$ is a total involution, and $A^{-1}$ is an involution
preinverse as well as an involute, so $A^{+}=A^{-1}$ since involution
preinverses are unique \textendash{} recall that the Moore-Penrose
inverse generalizes the ordinary matrix inverse.

The relationship between involution magmoids and corresponding semiheapoids
is briefly described in Appendix A.

\section{\label{sec:Prepoloids-and-related}Prepoloids and related magmoids}

The magmoids considered in this section are equipped with local units.
In the literature, the focus is on the special case when these magmoids
are magmas, namely semigroups, but here we generalize such magmas
to magmoids.

\subsection{\label{subsec:The-prepoloid-family}The prepoloid family}
\begin{defn}
\label{def7}Let $P$ be a semigroupoid, Then $P$ is 
\begin{enumerate}
\item a\emph{ prepoloid}\footnote{'Prepoloids' generalize the 'poloids' considered in Section 7. The
term 'poloid' for a generalized monoid was introduced and motivated
in \cite{key-9}.} when there is a local left unit $\lambda_{x}\in P$ and a local right
unit $\rho_{x}\in P$ for every $x\in P$\emph{;}
\item a \emph{pregroupoid}\footnote{Kock \cite{key-11} lets the term 'pregroupoid' refer to a set with
a partially defined \emph{ternary} operation.} when $P$ is a prepoloid such that for every $x\in P$ there is a
preinverse $x^{-1}\in P$ of $x$.
\end{enumerate}
This means that a prepoloid $P$ is a semigroupoid such that for every
$x\in P$ there are local units $\lambda{}_{x},\rho{}_{x}$ such that
$\left(\lambda{}_{x}x\right)=\left(x\rho{}_{x}\right)=x$. A pregroupoid
is a prepoloid $P$ such that for every $x\in P$ there is some $x^{-1}\in P$
such that $\left(xx^{-1}\right)\in\left\{ \lambda\right\} _{x}\cap\left\{ \rho\right\} _{x^{-1}}$
and $\left(x^{-1}x\right)\in\left\{ \rho\right\} _{x}\cap\left\{ \lambda\right\} _{x^{-1}}$.
\end{defn}

\subsubsection*{Prepoloids}
\begin{prop}
\label{pro71}\label{pro13}Let $P$ be a prepoloid, $x,y\in P$,
$\rho_{x}\in\left\{ \rho\right\} _{x}$ and $\lambda_{y}\in\left\{ \lambda\right\} _{y}$.
If $\rho_{x}=\lambda_{y}$ then $\left(xy\right)$.
\end{prop}
\begin{proof}
If $\rho_{x}=\lambda_{y}$ then $\left(\rho_{x}y\right)$ since $\left(\lambda_{y}y\right)$,
so $\left(\left(x\rho_{x}\right)y\right)=\left(xy\right)$ since $\left(x\rho_{x}\right)=x$.
\end{proof}
\begin{prop}
\label{pro72}\label{pro11} Let $P$ be a prepoloid with unique local
units, $x\in P$. Then $\left(\lambda_{x}\lambda_{x}\right)=\lambda_{x}=\lambda_{\lambda_{x}}=\rho{}_{\lambda_{x}}$
and $\left(\rho_{x}\rho_{x}\right)=\rho_{x}=\rho_{\rho_{x}}=\lambda{}_{\rho_{x}}$.
\end{prop}
\begin{proof}
If $x\in P$ then $x=\left(\lambda_{x}x\right)=\left(\lambda_{x}\left(\lambda_{x}x\right)\right)=\left(\left(\lambda_{x}\lambda_{x}\right)x\right)$,
so $\left(\lambda_{x}\lambda_{x}\right)\in\left\{ \lambda\right\} _{x}=\left\{ \lambda_{x}\right\} $.
Also, $\left(\lambda_{x}\lambda_{x}\right)=\lambda_{x}$ implies that
$\lambda_{x}\in\left\{ \lambda\right\} _{\lambda_{x}}=\left\{ \lambda_{\lambda_{x}}\right\} $
and $\lambda_{x}\in\left\{ \rho\right\} _{\lambda_{x}}=\left\{ \rho{}_{\lambda_{x}}\right\} $.
It is proved similarly that $\left(\rho_{x}\rho_{x}\right)=\rho_{x}=\rho_{\rho_{x}}=\lambda{}_{\rho_{x}}$.
\end{proof}
By Propositions \ref{pro51} and \ref{pro72}, an element of a prepoloid
with unique local units is thus a local unit if and only if it is
an idempotent.
\begin{prop}
\label{pro73}\label{pro12}Let $P$ be a prepoloid with unique local
units, $x,y\in P$. If $\left(xy\right)$ then $\lambda_{\left(xy\right)}=\lambda_{x}$
and $\rho_{\left(xy\right)}=\rho_{y}$.
\end{prop}
\begin{proof}
If $\left(xy\right)$ then $\left(xy\right)=\left(\left(\lambda_{x}x\right)y\right)=\left(\lambda_{x}\left(xy\right)\right)$,
so $\lambda_{x}\in\left\{ \lambda\right\} _{\left(xy\right)}=\left\{ \lambda_{\left(xy\right)}\right\} $.
Similarly, $\left(xy\right)=\left(x\left(y\rho_{y}\right)\right)=\left(\left(xy\right)\rho_{y}\right)$,
so $\rho_{y}\in\left\{ \rho\right\} _{\left(xy\right)}=\left\{ \rho_{\left(xy\right)}\right\} $.
\end{proof}
A prepoloid $P$ with unique local units can be equipped with unique
surjective functions 
\begin{gather*}
\mathfrak{s}:P\rightarrow\left\{ \lambda\right\} _{\!P},\qquad x\mapsto\lambda{}_{x},\\
\mathfrak{\mathfrak{t}}:P\rightarrow\left\{ \rho\right\} _{\!P},\qquad x\mapsto\rho{}_{x},
\end{gather*}
such that 
\begin{equation}
\left(\mathfrak{s}\left(x\right)\,x\right)=x,\quad\left(x\,\mathfrak{t}\left(x\right)\right)=x.\label{eq:lpre1}
\end{equation}

By Proposition \ref{pro72}, $\mathfrak{s}\left(\lambda_{x}\right)=\lambda_{x}$
for all $\lambda_{x}\in\left\{ \lambda\right\} _{\!P}$ and $\mathfrak{\mathfrak{t}}\left(\rho_{x}\right)=\rho_{x}$
for all $\rho_{x}\in\left\{ \rho\right\} _{\!P}$. A prepoloid with
unique local units can thus be regarded as a semigroupoid $\left(P,\mathfrak{m}\right)$
expanded to a prepoloid $\left(P,\mathfrak{m},\mathfrak{s},\mathfrak{t}\right)$
characterized by the uniqueness property and the identities (\ref{eq:l1}),
(\ref{eq:l2}), and (\ref{eq:lpre1}).

Note that there may exist more than one function $\mathfrak{s}:P\rightarrow\left\{ \lambda\right\} _{\!P}$
such that $\left(\mathfrak{s}\left(x\right)x\right)=x$ for all $x\in P$,
and more than one function $\mathfrak{t}:P\rightarrow\left\{ \rho\right\} _{\!P}$
such that $\left(x\mathfrak{t}\left(x\right)\right)=x$ for all $x\in P$
(see Example \ref{ex72}).

We call a prepoloid which admits not necessarily unique functions
$\mathfrak{s},\mathfrak{\mathfrak{t}}:P\rightarrow P$ satisfying
(\ref{eq:lpre1}) a \emph{bi-unital prepoloid}; those bi-unital prepoloids
which are magmas are \emph{bi-unital semigroups.} These are thus characterized
by the identities 
\begin{equation}
x\left(yx\right)=\left(xy\right)z,\quad\mathfrak{s}\left(x\right)\,x=x,\quad x\,\mathfrak{t}\left(x\right)=x.\label{s1}
\end{equation}

The class of bi-unital semigroups includes many types of semigroups
studied in the literature such as the function systems of Schweitzer
and Sklar \cite{key-15}, abundant semigroups \cite{key-4}, adequate
semigroups \cite{key-3}, Ehresmann semigroups \cite{key-12}, ample
semigroups and restriction semigroups (see, e.g., \cite{key-7}),
and regular and inverse semigroups with $\mathfrak{s}$ and $\mathfrak{t}$
defined by $\mathfrak{s}\left(x\right)=\left(xx^{-1}\right)$ and
$\mathfrak{t}\left(x\right)=\left(x^{-1}x\right)$.

As we have seen, several identities, such as $\mathfrak{s}\left(\mathfrak{s}\left(x\right)\right)=\mathfrak{s}\left(x\right)$,
$\left(\mathfrak{s}\left(x\right)\mathfrak{s}\left(x\right)\right)=\mathfrak{s}\left(x\right)$
and $\mathfrak{s}\left(xy\right)=\mathfrak{s}\left(x\right)$, can
be derived from the assumption that local units are unique. Bi-unital
pregroupoids and semigroups where this is not postulated can be required
to satisfy other conditions in order to have desirable properties;
various such requirements are used to characterize the bi-unital semigroups
found in the literature.

\subsubsection*{Pregroupoids (1)}

It is a well-known result in semigroup theory that a semigroup $S$
has a pseudoinverse $x^{\left(-1\right)}$ such that $xx^{\left(-1\right)}x=x$
for every $x\in S$ if and only if $S$ has a so-called generalized
inverse $x^{-1}$ such that $xx^{-1}x=x$ and $x^{-1}xx^{-1}=x^{-1}$
for every $x\in S$. A regular semigroup can thus be defined by either
condition. This result can be generalized to semigroupoids.
\begin{prop}
Let $P$ be a semigroupoid. Each $x\in P$ has a preinverse $x^{-1}\in P$
if and only if each $x\in P$ has a pseudoinverse $x^{\left(-1\right)}\in P$. 
\end{prop}
\begin{proof}
Trivially, each preinverse $x^{-1}$ of $x$ is a pseudoinverse $x^{\left(-1\right)}$
of $x$. Conversely, if $\overline{x}$ is a pseudoinverse of $x$
then $\left(x\overline{x}x\right)$, $\left(x\overline{x}\right)$
and $\left(\overline{x}x\right)$. Thus, $\left(\overline{x}x\overline{x}\right)$,
$\left(x\overline{x}x\overline{x}x\right)$ and $\left(\overline{x}x\overline{x}x\overline{x}x\overline{x}\right)$,
and as $\left(x\overline{x}x\right)=x$ we have 
\begin{gather*}
\left(x\left(\overline{x}x\overline{x}\right)x\right)=\left(x\overline{x}x\overline{x}x\right)=\left(x\overline{x}x\right)=x,\\
\left(\left(\overline{x}x\overline{x}\right)x\left(\overline{x}x\overline{x}\right)\right)=\left(\overline{x}x\overline{x}x\overline{x}x\overline{x}\right)=\left(\overline{x}x\overline{x}x\overline{x}\right)=\left(\overline{x}x\overline{x}\right),
\end{gather*}
 so $x\,\mathbf{I}\,\left(\overline{x}x\overline{x}\right)$, meaning
that $\left(\overline{x}x\overline{x}\right)\in\mathbf{I}\left\{ \right\} _{x}$.
\end{proof}
Recall that an inverse semigroup can be defined as a regular semigroup
whose preinverses are unique, or as a regular semigroup whose idempotents
commute. These two characterizations are equivalent also when inverse
semigroups are generalized to pregroupoids.
\begin{lem}
\label{lem1-1}Let $P$ be a pregroupoid with unique preinverses,
$i,j\in P$. If $i,j$ are idempotents and $\left(ij\right)$ then
$\left(ij\right)$ and $\left(ji\right)$ are idempotents.
\end{lem}
\begin{proof}
We have $\left(\left(ij\right)\left(ij\right)^{-1}\left(ij\right)\right)$,
so $\left(ij\left(ij\right)^{-1}ij\right)$, so $\left(j\left(ij\right)^{-1}i\right)$.
Since $\left(ii\right)$, $\left(jj\right)$ and $\left(ij\right)$,
$\left(j\left(ij\right)^{-1}i\right)$ implies $\left(ijj\left(ij\right)^{-1}i\right)$
and $\left(j\left(ij\right)^{-1}iij\right)$. Thus, $\left(ijj\left(ij\right)^{-1}iij\right)$
and $\left(j\left(ij\right)^{-1}iijj\left(ij\right)^{-1}i\right)$,
and we have
\begin{gather*}
\left(\left(j\left(ij\right)^{-1}i\right)\left(ij\right)\left(j\left(ij\right)^{-1}i\right)\right)=\left(\left(j\left(ij\right)^{-1}i\right)\left(j\left(ij\right)^{-1}i\right)\right)=\\
\left(j\left(\left(ij\right)^{-1}\left(ij\right)\left(ij\right)^{-1}\right)i\right)=\left(j\left(ij\right)^{-1}i\right),\\
\left(\left(ij\right)\left(j\left(ij\right)^{-1}i\right)\left(ij\right)\right)=\left(\left(ij\right)\left(ij\right)^{-1}\left(ij\right)\right)=\left(ij\right),
\end{gather*}
so $\left(j\left(ij\right)^{-1}i\right)$ is an idempotent and $\left(j\left(ij\right)^{-1}i\right)\,\mathfrak{\mathbf{I}}\,\left(ij\right)$.
Using the fact that preinverses are unique so that idempotents are
preinverses of themselves by Corollary \ref{cor52}, we conclude that
$\left(ij\right)$ is an idempotent since $\left(ij\right)=\left(j\left(ij\right)^{-1}i\right)^{-1}=\left(j\left(ij\right)^{-1}i\right)$. 

Finally, if $\left(ij\right)$ is an idempotent then $\left(\left(ij\right)\left(ij\right)\right)$,
so $\left(ji\right)$, and it can be shown in the same way as for
$\left(ij\right)$ that $\left(ji\right)$ is an idempotent.
\end{proof}
\begin{prop}
\label{pro75}Let $P$ be a pregroupoid. Then $P$ has a unique preinverses
if and only if $\left(ij\right)=\left(ji\right)$ for all idempotents
$i,j\in P$ such that $\left(ij\right)$.
\end{prop}
\begin{proof}
Let $P$ have unique preinverses. If $i,j\in P$ are idempotents and
$\left(ij\right)$ then $\left(ij\right)$ and $\left(ji\right)$
are idempotents by Lemma \ref{lem1-1}. Thus, 
\[
\left(\left(ij\right)\left(ji\right)\left(ij\right)\right)=\left(\left(ij\right)\left(ij\right)\right)=\left(ij\right),\quad\left(\left(ji\right)\left(ij\right)\left(ji\right)\right)=\left(\left(ji\right)\left(ji\right)\right)=\left(ji\right).
\]
so $\left(ij\right)\,\mathbf{I}\,\left(ji\right)$. Hence, $\left(ji\right)\in\mathbf{I}\left\{ \right\} _{\left(ij\right)}=\left\{ \left(ij\right)^{-1}\right\} =\left\{ \left(ij\right)\right\} $,
using Corollary \ref{cor52}.

Conversely, let idempotents in $P$ commute. If $y$ and $y'$ are
preinverses of $x$, then $\left(xyx\right)=x$ and $\left(xy'\!x\right)=x$
so that $\left(yx\right)$ and $\left(y'\!x\right)$, so $\left(yx\right)$
and $\left(y'\!x\right)$ are idempotents since $\left(yxyx\right)=\left(yx\right)$
and $\left(y'\!xy'\!x\right)=\left(y'\!x\right).$ Similarly, $\left(xy\right)$
and $\left(xy'\right)$ are idempotents. We also have $\left(yxy\right)=y$
and $\left(y'\!xy'\right)=y'$. Thus, $y\!=\!\left(yxy\right)=\left(yxy'\!xy\right)=\left(y'\!xyxy\right)=\left(y'\!xy\right)=\left(y'\!xy'\!xy\right)=\left(y'\!xyxy'\right)=\left(y'\!xy'\right)\!=\!y'.$
\end{proof}

\subsubsection*{Pregroupoids (2)}
\begin{prop}
Let $P$ be\emph{ }a pregroupoid $P$ with unique preinverses, $x\in P$.
Then $\lambda_{x}=\lambda_{x}^{-1}$ and $\rho_{x}=\rho_{x}^{-1}$.
\end{prop}
\begin{proof}
As $\lambda_{x}$ and $\rho_{x}$ are idempotents by Proposition \ref{pro72},
we have $\lambda_{x}=\lambda_{x}^{-1}$ and $\rho_{x}=\rho_{x}^{-1}$
by Corollary \ref{cor52}.
\end{proof}
\begin{prop}
\label{pro77}Let $P$ be a pregroupoid with unique preinverses, $x\in P$.
If $x^{-1}$ is the preinverse of $x$ then $\left(x^{-1}\right)^{\!-1}=x$.
\end{prop}
\begin{proof}
If $x\,\mathbf{I}\,x^{-1}$ then $x^{-1}\,\mathbf{I}\,x$, so $x\in\mathbf{I}\left\{ \right\} _{x^{-1}}=\left\{ \left(x^{-1}\right)^{-1}\right\} $.
\end{proof}
\begin{prop}
\label{pro78}Let $P$ be a pregroupoid with unique preinverses, $x,y\in P$.
If $\left(xy\right)$ then $\left(xy\right)^{-1}=\left(y^{-1}x^{-1}\right)$.
\end{prop}
\begin{proof}
We have $\left(xy\right)$, $\left(xx^{-1}\right)$, $\left(x^{-1}x\right)$,
$\left(yy^{-1}\right)$ and $\left(y^{-1}y\right)$. Hence, by Proposition
\ref{pro75} and the fact that $\left(x^{-1}x\right)$ and $\left(yy^{-1}\right)$
are idempotents,
\[
\left(xy\right)=\left(xx^{-1}xyy^{-1}y\right)=\left(xyy^{-1}x^{-1}xy\right)
\]
so that $\left(y^{-1}x^{-1}\right)$, and
\[
\left(y^{-1}x^{-1}\right)=\left(y^{-1}yy^{-1}x^{-1}xx^{-1}\right)=\left(y^{-1}x^{-1}xyy^{-1}x^{-1}\right),
\]
so $\left(xy\right)\,\mathbf{I}\,\left(y^{-1}x^{-1}\right)$, so $\left(y^{-1}x^{-1}\right)\in\mathbf{I}\left\{ \right\} _{\left(xy\right)}=\left\{ \left(xy\right)^{-1}\right\} $.
\end{proof}
If $P$ is a pregroupoid with unique preinverses then $\left(x^{-1}\right)^{-1}=x$,
so $P$ can be equipped with a unique bijection
\[
\mathfrak{i}:P\rightarrow P,\qquad x\mapsto x^{-1}
\]
 such that, for all $x\in P$,
\begin{gather}
\left(\left(x\,\mathfrak{i}\left(x\right)\right)\,x\right)=x,\quad\left(x\,\left(\mathfrak{i}\left(x\right)\,x\right)\right)=x,\label{eq:lpre2}\\
\left(\left(\mathfrak{i}\left(x\right)\,x\right)\,\mathfrak{i}\left(x\right)\right)=\mathfrak{i}\left(x\right),\quad\left(\mathfrak{i}\left(x\right)\,\left(x\,\mathfrak{i}\left(x\right)\right)\right)=\mathfrak{i}\left(x\right).\nonumber 
\end{gather}

A pregroupoid with unique preinverses can thus be regarded as a prepoloid
$\left(P,\mathfrak{m},\mathfrak{s},\mathfrak{t}\right)$ expanded
to a pregroupoid $\left(P,\mathfrak{m},\mathfrak{i,\mathfrak{s},\mathfrak{t}}\right)$
characterized by the uniqueness property and the identities (\ref{eq:l1}),
(\ref{eq:l2}), (\ref{eq:lpre1}), and (\ref{eq:lpre2}).

If a pregroupoid with unique preinverses \textendash{} or alternatively
its reduct $\left(P,\mathfrak{m},\mathfrak{i}\right)$ \textendash{}
is a magma, it is thus an inverse semigroup, characterized by the
uniqueness of the preinverses and the identities
\begin{equation}
x\left(yz\right)=\left(xy\right)z,\quad xx^{-1}x=x,\quad x^{-1}xx^{-1}=x^{-1}.\label{eq:s2}
\end{equation}

Note that a pregroupoid with unique preinverses does not necessarily
have unique local left and right units. While $\left(xx^{-1}\right)$
and $\left(x^{-1}x\right)$ are uniquely determined by $x$ when its
preinverse $x^{-1}$ is unique, and are local left and local right
units, respec\-tively, for $x$, $\left(xx^{-1}\right)$ and $\left(x^{-1}x\right)$
are not necessarily the only local units for $x$.
\begin{example}
\label{ex72}Let $S$ be a set $\left\{ x,y\right\} $ with a binary
operation given by the table
\[
\begin{array}{ccc}
 & x & y\\
x & x & x\\
y & x & y.
\end{array}
\]
Let $\alpha,\beta,\gamma\in\left\{ x,y\right\} $. If $\alpha=\beta=\gamma=y$
then $\left(\alpha\beta\right)\gamma=\alpha\left(\beta\gamma\right)=y$;
otherwise, $\left(\alpha\beta\right)\gamma=\alpha\left(\beta\gamma\right)=x$.
Thus, in all cases $\left(\alpha\beta\right)\gamma=\alpha\left(\beta\gamma\right)$,
so $S$ is a semigroup. In particular, $xxx=x$ and $yyy=y$, so $x$
is a preinverse of $x$ and $y$ is a preinverse of $y$. Also, $yxy\neq y$,
so $x$ is not a preinverse of $y$, and $y$ is not a preinverse
of $x$. Hence, $S$ is a pregroupoid where preinverses are unique. 

Local units are not unique, however; we have $xx=x$, so $xx^{-1}=x$
is a local left unit for $x$, but $y$ is also a local left unit
for $x$ since $yx=x$. As a consequence, if we set $\mathfrak{s}_{1}\left(x\right)=x,\mathfrak{s}_{1}\left(y\right)=y,\mathfrak{s}_{2}\left(x\right)=y,\mathfrak{s}_{2}\left(y\right)=y$
then $\left(\mathfrak{s}_{1}\left(x\right)x\right)=\left(\mathfrak{s}_{2}\left(x\right)x\right)=x$
and $\left(\mathfrak{s}_{1}\left(y\right)y\right)=\left(\mathfrak{s}_{2}\left(y\right)y\right)=y$
but $\mathfrak{s}_{1}\neq\mathfrak{s}_{2}$.
\end{example}

\subsubsection*{Pregroupoids as prepoloids}

Let $\left(P,\mathfrak{m},\mathfrak{i}\right)$ be a reduct of the
pregroupoid $\left(P,\mathfrak{m},\mathfrak{i,\mathfrak{s},\mathfrak{t}}\right)$.
One can expand $\left(P,\mathfrak{m},\mathfrak{i}\right)$ to a pregroupoid
$\left(P,\mathfrak{m},\mathfrak{i},\mathbf{s},\mathbf{t}\right)$
by setting $\mathbf{s}\left(x\right)=\left(x\,\mathfrak{i}\left(x\right)\right)$
and $\mathbf{t}\left(x\right)=\left(\mathfrak{i}\left(x\right)\,x\right)$.
Finally, we obtain the prepoloid $\left(P,\mathfrak{m},\mathbf{s},\mathbf{t}\right)$
as a reduct of $\left(P,\mathfrak{m},\mathfrak{i,\mathbf{s},\mathbf{t}}\right)$;
below we note two useful results about this prepoloid in the case
when $\mathbf{s}$ and $\mathbf{t}$ are unique. Recall from Section
\ref{subsec:Canonical} that $\boldsymbol{\uplambda}_{x}=\left(xx^{-1}\right)$
and $\boldsymbol{\uprho}_{x}=\left(x^{-1}x\right)$.
\begin{prop}
\label{pro79}Let $P$ be a pregroupoid with unique canonical local
units, $x\in P$. Then $\left(\boldsymbol{\uplambda}_{x}\boldsymbol{\uplambda}_{x}\right)=\boldsymbol{\uplambda}_{x}=\boldsymbol{\uplambda}_{\boldsymbol{\uplambda}_{x}}=\boldsymbol{\uprho}{}_{\boldsymbol{\uplambda}_{x}}$
and $\left(\boldsymbol{\uprho}_{x}\boldsymbol{\uprho}_{x}\right)=\boldsymbol{\uprho}_{x}=\boldsymbol{\uprho}_{\boldsymbol{\uprho}_{x}}=\boldsymbol{\uplambda}{}_{\boldsymbol{\uprho}_{x}}$.
\end{prop}
\begin{proof}
Analogous to the proof of Proposition \ref{pro72}, in addition using
the facts that $\left(\left(xx^{-1}\right)\left(xx^{-1}\right)\right)=\left(xx^{-1}\right)$
and $\left(\left(x^{-1}x\right)\left(x^{-1}x\right)\right)=\left(x^{-1}x\right)$.
\end{proof}
\begin{prop}
\label{pro710}Let $P$ be a pregroupoid with unique canonical local
units,\linebreak{}
 $x,y\in P$. If $\left(xy\right)$ then $\boldsymbol{\uplambda}_{\left(xy\right)}=\boldsymbol{\uplambda}_{x}$
and $\boldsymbol{\uprho}_{\left(xy\right)}=\boldsymbol{\uprho}_{y}$.
\end{prop}
\begin{proof}
Analogous to the proof of Proposition \ref{pro73}.
\end{proof}

\subsection{\label{subsec:The-one-sided-skew-prepoloid}The skew-prepoloid family}
\begin{defn}
\label{d72}Let $P$ be a left (resp. right) semigroupoid. Then $P$
is 
\begin{enumerate}
\item a \emph{left (resp. right) skew-prepoloid} when there is a local left
(resp. right) unit $\lambda_{x}\in P$ (resp. $\rho_{x}\in P$) for
every $x\in P$\emph{;}
\item a \emph{left (resp. right) skew-pregroupoid} when $P$ is a left (resp.
right) skew-prepoloid such that for each $x\in P$ there is a right
(resp. left) preinverse $x^{-1}$ of $x$ in $P$.
\end{enumerate}
In view of the left-right duality of these notions, it suffices to
consider left skew-prepoloids and left skew-pregroupoids here. 

By Definition \ref{d72}, a left skew-prepoloid is a left semigroupoid
$P$ such that for every $x\in P$ there is some $\lambda{}_{x}\in P$
such that $\left(\lambda{}_{x}x\right)=x$. A left skew-groupoid is
a left skew-poloid $P$ such that for every $x\in P$ there is some
$x^{-1}\in P$ such that $\left(xx^{-1}\right)\!\in\!\left\{ \lambda\right\} _{x}$
and $\left(x^{-1}x\right)\!\in\!\left\{ \lambda\right\} _{x^{-1}}$,
so that $\left(\left(xx^{-1}\right)x\right)\!=\!x$ and $\left(\left(x^{-1}x\right)x^{-1}\right)\!=\!x^{-1}$. 
\end{defn}

\subsubsection*{Left skew-prepoloids}
\begin{prop}
\label{pro711}Let $P$ be a left skew-prepoloid, $x,y\in P$. If
$\lambda_{y}$ is a local left unit for $y$ and $\left(x\lambda_{y}\right)$
then $\left(xy\right)$.
\end{prop}
\begin{proof}
If $\left(x\lambda_{y}\right)$ then $\left(x\left(\lambda_{y}y\right)\right)=\left(xy\right)$
since $\left(\lambda_{y}y\right)=y$.
\end{proof}
\begin{prop}
\label{pro714}Let $P$ be\emph{ }a left skew-prepoloid $P$ with
unique local units, $x\in P$. Then $\left(\lambda_{x}\lambda_{x}\right)=\lambda_{x}=\lambda_{\lambda_{x}}$.
\end{prop}
\begin{proof}
We have $x=\left(\lambda_{x}x\right)=\left(\lambda_{x}\left(\lambda_{x}x\right)\right)=\left(\left(\lambda_{x}\lambda_{x}\right)x\right)$,
so $\left(\lambda_{x}\lambda_{x}\right)\in\left\{ \lambda\right\} _{x}=\left\{ \lambda_{x}\right\} $,
so $\left(\lambda_{x}\lambda_{x}\right)=\lambda_{x}$, so $\lambda_{x}\in\left\{ \lambda\right\} _{\lambda_{x}}=\left\{ \lambda_{\lambda_{x}}\right\} $. 
\end{proof}
\begin{prop}
\label{pro715}Let $P$ be\emph{ }a left skew-prepoloid $P$ with
unique local units, \linebreak{}
$x,y\in P$. If $\left(xy\right)$ then $\lambda_{\left(xy\right)}=\lambda_{x}$.
\end{prop}
\begin{proof}
If $\left(xy\right)$ then $\left(xy\right)=\left(\left(\lambda_{x}x\right)y\right)=\left(\lambda_{x}\left(xy\right)\right)$,
so $\lambda_{x}\in\left\{ \lambda\right\} _{\left(xy\right)}=\left\{ \lambda_{\left(xy\right)}\right\} $.
\end{proof}
By Definition \ref{d72}, every left skew-prepoloid $P$ with unique
local left units can be equipped with a unique surjective function
\begin{gather*}
\mathfrak{s}:P\rightarrow\left\{ \lambda\right\} _{\!P},\qquad x\mapsto\lambda{}_{x}
\end{gather*}
 such that, for all $x\in P$,
\begin{equation}
\left(\mathfrak{s}\left(x\right)\,x\right)=x.\label{eq:ll3}
\end{equation}

By Proposition (\ref{pro714}), $\mathfrak{s}\left(\lambda{}_{x}\right)=\lambda{}_{x}$
for all $\lambda_{x}\in\left\{ \lambda\right\} _{\!P}$.

A left skew-poloid with unique local units can thus be regarded as
a left semigroupoid $\left(P,\mathfrak{m}\right)$ expanded to a left
skew-poloid $\left(P,\mathfrak{m},\mathfrak{s}\right)$ characterized
by the uniqueness property and the identities (\ref{eq:l1}) and (\ref{eq:ll3}). 

We call a left semigroupoid $P$ which admits a not necessarily unique
function $\mathfrak{s}:P\rightarrow P$ satisfying (\ref{eq:ll3})
a \emph{left unital semigroupoid}. If $P$ is a magma then $\left(P,\mathfrak{m},\mathfrak{s}\right)$
is a semigroup such that, for all $x,y,z\in P$,
\begin{equation}
x\left(yz\right)=\left(xy\right)z,\quad\mathfrak{s}\left(x\right)x=x.\label{eq:s3}
\end{equation}
Such a semigroup may be called a \emph{left unital semigroup}. The
class of left unital semigroups includes many types of semigroups
studied in the literature, for example, $D$-semigroups \cite{key-16},
left abundant semigroups, left adequate semigroups, left Ehresmann
semigroups, left ample semigroups and left restriction semigroups.

\subsubsection*{Left skew-pregroupoids}
\begin{prop}
Let $P$ be\emph{ }a left skew-pregroupoid with unique right preinverses,
$x\in P$. Then $\lambda_{x}^{-1}=\lambda_{x}$. 
\end{prop}
\begin{proof}
As $\lambda_{x}$ is an idempotent by Proposition \ref{pro714}, we
have $\lambda_{x}\in\mathbf{I}^{+}\!\left\{ \right\} _{x}=\left\{ \lambda_{x}^{-1}\right\} $
as in Corollary \ref{cor52}.
\end{proof}
\begin{prop}
Let $P$ be a left skew-pregroupoid with unique right preinverses,
$x\in P$. If $x^{-1}$ is the right preinverse of $x$ then $x=\left(x^{-1}\right)^{-1}$. 
\end{prop}
\begin{proof}
If $x\,\mathbf{I}^{+}\,x^{-1}$ then $x^{-1}\,\mathbf{I}^{+}\,x$,
so $x\in\mathbf{I}^{+}\!\left\{ \right\} _{x^{-1}}=\left\{ \left(x^{-1}\right)^{-1}\right\} $.
\end{proof}
If $P$ is a left skew-pregroupoid with unique right preinverses then
$\left(x^{-1}\right)^{-1}=x$, so $P$ can be equipped with a unique
bijection
\[
\mathfrak{i}:P\rightarrow P,\qquad x\mapsto x^{-1}
\]
 such that, for all $x\in P$,
\begin{equation}
\left(\left(x\,\mathfrak{i}\left(x\right)\right)\,x\right)=x,\quad\left(\left(\mathfrak{i}\left(x\right)\,x\right)\,\mathfrak{i}\left(x\right)\right)=\mathfrak{i}\left(x\right).\label{eq:pre4}
\end{equation}

A left skew-pregroupoid with unique right inverses can thus be regarded
as a left skew-poloid $\left(P,\mathfrak{m},\mathfrak{s}\right)$
expanded to a left skew-groupoid $\left(P,\mathfrak{m},\mathfrak{i},\mathfrak{s}\right)$
characterized by the uniqueness property and the identities (\ref{eq:l1}),
(\ref{eq:ll3}) and (\ref{eq:pre4}). 

Those left skew-pregroupoids with unique right preinverses which are
magmas are again just inverse semigroups, characterized by the uniqueness
of preinverses and the identities (\ref{eq:s2}).

\subsubsection*{Skew-pregroupoids as skew-prepoloids}

Let $\left(P,\mathfrak{m},\mathfrak{i}\right)$ be a reduct of $\left(P,\mathfrak{m},\mathfrak{i,\mathfrak{s}}\right)$,
and expand $\left(P,\mathfrak{m},\mathfrak{i}\right)$ to a left skew-pregroupoid
$\left(P,\mathfrak{m},\mathfrak{i},\mathbf{s}\right)$ by setting
$\mathbf{s}\left(x\right)=\left(x\,\mathfrak{i}\left(x\right)\right)$.
We note two useful results about the left skew-prepoloid $\left(P,\mathfrak{m},\mathbf{s}\right)$,
obtained as a reduct of $\left(P,\mathfrak{m},\mathfrak{i},\mathbf{s}\right)$,
in the case when $\mathbf{s}$ is unique. Recall that $\boldsymbol{\uplambda}_{x}=\left(xx^{-1}\right)$.
\begin{prop}
\label{pro720}Let $P$ be\emph{ }a left skew-pregroupoid with unique
canonical local left units, $x\in P$. Then $\left(\boldsymbol{\uplambda}_{x}\boldsymbol{\uplambda}_{x}\right)=\boldsymbol{\uplambda}_{x}=\boldsymbol{\uplambda}_{\boldsymbol{\uplambda}{}_{x}}$.
\end{prop}
\begin{proof}
Analogous to the proof of Proposition \ref{pro714}, although note
that the present proof requires the fact that $\left(\left(xx^{-1}\right)\left(xx^{-1}\right)\right)=\left(\left(\left(xx^{-1}\right)x\right)x^{-1}\right)=\left(xx^{-1}\right)$
so that $\left(\boldsymbol{\uplambda}_{x}\boldsymbol{\uplambda}_{x}\right)\in\left\{ \boldsymbol{\uplambda}\right\} _{x}$.
\end{proof}
\newpage{}
\begin{prop}
\label{pro721-1}Let $P$ be\emph{ }a left skew-pregroupoid with unique
canonical local left units, $x,y\in P$. If $\left(xy\right)$ then
$\boldsymbol{\uplambda}_{x}=\boldsymbol{\uplambda}_{\left(xy\right)}$.
\end{prop}
\begin{proof}
Analogous to the proof of Proposition \ref{pro715}.
\end{proof}

\section{\label{sec:Poloids-and-related}Poloids and related magmoids}

The magmoids considered in this section differ from those in Section
6 in that they are equipped with two-sided or twisted units rather
than local units. These magmoids are, roughly speaking, categories
and some of their specializations and generalizations, considered
as algebraic structures. As noted in the introduction to this article,
categories and groupoids are indeed examples of important notions
involving partial binary operations.

\subsection{\label{subsec:The-poloid-family}The poloid family}
\begin{defn}
\label{def6-1}Let $P$ be a semigroupoid. Then $P$ is
\begin{enumerate}
\item a \emph{poloid} when there is a left effective unit $\ell_{x}\in P$
and a right effective unit $r_{x}\in P$ for every $x\in P$\emph{; }
\item a \emph{groupoid} when $P$ is a poloid such that for every $x\in P$
there is a strong preinverse $x^{-1}\in P$ of $x$. 
\end{enumerate}
\end{defn}
More explicitly, a poloid is a semigroupoid $P$ such that for every
$x\in P$ there are two-sided units $\ell_{x},r_{x}$ such that $\left(\ell_{x}x\right)=\left(xr_{x}\right)=x$,
and also $\left(y\ell_{x}\right)=\left(yr_{x}\right)=\left(\ell_{x}y\right)=\left(r_{x}y\right)=y$
for all $x,y\in P$ such that, respectively, $\left(y\ell_{x}\right)$,
$\left(yr_{x}\right)$, $\left(\ell_{x}y\right)$, and $\left(r_{x}y\right)$.
A groupoid is a poloid $P$ such that for every $x\in P$ there is
some $x^{-1}\in P$ such that there are two-sided units $\ell_{x},r_{x}$
such that $\left(xx^{-1}\right)=\ell_{x}=r_{x^{-1}}$ and $\left(x^{-1}x\right)=r_{x}=\ell_{x^{-1}}$,
so that $\left(xx^{-1}x\right)=x$, $\left(x^{-1}xx^{-1}\right)=x^{-1}$,
and also $\left(yxx^{-1}\right)=\left(yx^{-1}x\right)=\left(xx^{-1}y\right)=\left(x^{-1}xy\right)=y$
for all $x,y\in P$ such that, respectively, $\left(y\left(xx^{-1}\right)\right)$,
$\left(y\left(x^{-1}x\right)\right)$, $\left(\left(xx^{-1}\right)y\right)$,
and $\left(\left(x^{-1}x\right)y\right)$.

It is shown below that every element of a poloid has unique left and
right effective units and a unique preinverse, and it turns out that
poloids have much in common with prepoloids with unique local units
and preinverses. A similar remark applies to skew-poloids, defined
below, in relation to skew-prepoloids.

A poloid is just a (small) category regarded as an abstract algebraic
structure \cite{key-9}, while a groupoid is a (small) category with
preinverses, also regarded as an abstract algebraic structure. While
categories are usually defined in another way, definitions similar
to the definition of poloids given here can also be found in the literature.
For example, Ehresmann \cite{key-2} proposed the following definition:
\begin{quotation}
{\small{}Eine Kategorie ist eine Klasse $C$ von Elementen, in der
eine Multiplikation gegeben ist $\left(f,g\right)\rightarrow fg$
für gewisse Paare $(f,g)$ von Elementen von $C$, welche folgenden
Axiomen genügt: }{\small \par}

{\small{}1. Wenn $h\left(fg\right)$ oder $\left(hf\right)g$ definiert
ist, dann sind die beide Elemente definiert und $h\left(fg\right)=\left(hf\right)g$.}{\small \par}

{\small{}2. Wenn $hf$ und $fg$ definiert sind, dann ist auch $h\left(fg\right)$
definiert.}{\small \par}

{\small{}Ein Element $e$ von $C$ wird eine Einheit genannt, falls
$fe=f$ und $eg=g$ für alle Elemente $f$ und $g$ von $C$ ist,
für welche $fe$ und $eg$ definiert sind.}{\small \par}

{\small{}3. Für jedes $f\in C$ gibt es zwei Einheiten $\alpha\left(f\right)$
und $\beta\left(f\right)$, so dass $f\alpha\left(f\right)$ und $\beta\left(f\right)f$
definiert sind. (p. 50).}{\small \par}
\end{quotation}
Proposition \ref{prop4} below implies that if a poloid is a semigroup
then it is a monoid, since it has only one two-sided unit, denoted
$1$, and if a groupoid is a semigroup then it is a group. Conversely,
a poloid with just one two-sided unit is a monoid, and a groupoid
with just one two-sided unit is a group \cite{key-9}. A groupoid
is thus a generalized group, as expected, while poloids generalize
groups indirectly and in two ways, via monoids and via groupoids.\newpage{}

\subsubsection*{Poloids}
\begin{prop}
\label{prop4}Let $P$ be a poloid. If $e,e'\in\left\{ e\right\} _{\!P}$
and $\left(ee'\right)$ then $e=e'$.
\end{prop}
\begin{proof}
We have $e=\left(ee'\right)=e'$.
\end{proof}
\begin{prop}
\label{pro2}Let $P$ be a poloid, $x\in P$. Then there is a unique
left effective unit $\ell_{x}\in P$ and a unique right effective
unit $r_{x}\in P$ for every $x\in P$.
\end{prop}
\begin{proof}
If $\ell_{x},\ell_{x}'\in\left\{ \ell\right\} _{x}$ then $x=\left(\ell_{x}x\right)=\left(\ell_{x}\left(\ell_{x}'x\right)\right)=\left(\left(\ell_{x}\ell_{x}'\right)x\right)$,
so $\left(\ell_{x}\ell_{x}'\right)$ and thus $\ell_{x}=\ell_{x}'$.
Dually, if $r_{x},r_{x}'\in\left\{ r\right\} _{x}$ then $x=\left(xr_{x}\right)=\left(\left(xr_{x}'\right)r_{x}\right)=\left(x\left(r_{x}'r_{x}\right)\right)$,
so $\left(r_{x}'r_{x}\right)$ and thus $r_{x}=r_{x}'$.
\end{proof}
\begin{prop}
\label{pro4}Let $P$ be a poloid, $x,y\in P$. Then $\left(xy\right)$
if and only if $r_{x}=\ell_{y}$.
\end{prop}
\begin{proof}
If $\left(xy\right)$ then $\left(xy\right)=\left(\left(xr_{x}\right)\left(\ell_{y}y\right)\right)=\left(\left(\left(xr_{x}\right)\ell_{y}\right)y\right)=\left(\left(x\left(r_{x}\ell_{y}\right)\right)y\right)$,
so $\left(r_{x}\ell_{y}\right)$, so $r_{x}=\ell{}_{y}$. Conversely,
if $r_{x}=\ell{}_{y}$ then $\left(x\ell_{y}\right)$, so $\left(x\left(\ell_{y}y\right)\right)$
and thus $\left(x\left(\ell_{y}y\right)\right)=\left(xy\right)$ since
$\left(\ell_{y}y\right)=y$. 
\end{proof}
\begin{prop}
\label{pro53}Let $P$ be a poloid. If $e\in\left\{ e\right\} _{\!P}$
then $\left(ee\right)=e=\ell_{e}=r_{e}$.
\end{prop}
\begin{proof}
We have $e=\left(\ell{}_{e}e\right)=\ell{}_{e}$ and $e=\left(er_{e}\right)=r_{e}$.
\end{proof}
\begin{cor}
\label{cor53}Let $P$ be a poloid, $x\in P$. Then $\left(\ell{}_{x}\ell{}_{x}\right)=\ell_{x}=\ell_{\ell_{x}}=r{}_{\ell_{x}}$
and $\left(r_{x}r_{x}\right)=r_{x}=\ell{}_{r_{x}}=r_{r_{x}}$.
\end{cor}
\begin{prop}
\label{pro6-1}Let $P$ be a poloid, $x,y\in P$. If $\left(xy\right)$
then $\ell_{\left(xy\right)}\!=\!\ell_{x}$ and $r_{\left(xy\right)}\!=\!r_{y}$.
\end{prop}
\begin{proof}
We use the fact that effective units are unique by Proposition \ref{pro2}.
If $\left(xy\right)$ then $\left(xy\right)=\left(\left(\ell_{x}x\right)y\right)=\left(\ell_{x}\left(xy\right)\right)$,
so $\ell_{x}\in\left\{ \ell\right\} _{\left(xy\right)}=\left\{ \ell_{\left(xy\right)}\right\} $.
Dually, $\left(xy\right)=\left(x\left(yr_{y}\right)\right)=\left(\left(xy\right)r_{y}\right)$,
so $r_{y}\in\left\{ r\right\} _{\left(xy\right)}=\left\{ r_{\left(xy\right)}\right\} $.
\end{proof}
In view of Definition (\ref{def6-1}) and Proposition \ref{pro2},
every poloid $P$ can be equipped with unique surjective functions
\begin{gather*}
\mathfrak{s}:P\rightarrow\left\{ \ell\right\} _{\!P},\qquad x\mapsto\ell_{x},\\
\mathfrak{t}:P\rightarrow\left\{ r\right\} _{\!P},\qquad x\mapsto r_{x}
\end{gather*}
such that, for any $x,y\in P$,
\begin{gather}
\left(\mathfrak{s}\left(x\right)\,x\right)=x,\quad\left(x\,\mathfrak{t}\left(x\right)\right)=x\nonumber \\
\left(\mathfrak{s}\left(x\right)\,y\right)\rightarrow\left(\mathfrak{s}\left(x\right)\,y\right)=y,\quad\left(\mathfrak{t}\left(x\right)\,y\right)\rightarrow\left(\mathfrak{t}\left(x\right)\,y\right)=y,\label{eq:l3}\\
\left(y\,\mathfrak{s}\left(x\right)\right)\rightarrow\left(y\,\mathfrak{s}\left(x\right)\right)=y,\quad\left(y\,\mathfrak{t}\left(x\right)\right)\rightarrow\left(y\,\mathfrak{t}\left(x\right)\right)=y.\nonumber 
\end{gather}

By Corollary \ref{cor53}, $\mathfrak{s}\left(\ell_{x}\right)=\ell_{x}$
for all $\ell_{x}\in\left\{ \ell\right\} _{\!P}$ and $\mathfrak{t}\left(r_{x}\right)=r_{x}$
for all $r_{x}\in\left\{ r\right\} _{\!P}$.

A poloid $P$ can thus be regarded as an expansion $\left(P,\mathfrak{m},\mathfrak{s},\mathfrak{t}\right)$,
characterized by the identities (\ref{eq:l1}), (\ref{eq:l2}), and
(\ref{eq:l3}), of a semigroupoid $\left(P,\mathfrak{m}\right)$.

If $\left(P,\mathfrak{m},\mathfrak{s},\mathfrak{t}\right)$ is a magma
then it degenerates to a monoid $\left(P,\mathfrak{m},1\right)$ where
$\mathfrak{s}\left(x\right)=\mathfrak{t}\left(x\right)=1$ for all
$x\in P$.

\subsubsection*{Groupoids}
\begin{prop}
\label{pro9-1}Let $P$ be a poloid, $x\in P$. If $x^{\left(-1\right)}$
is a strong pseudoinverse of $x$ then $x^{\left(-1\right)}$ is a
strong preinverse of $x$.
\end{prop}
\begin{proof}
For any strong pseudoinverse $x^{\left(-1\right)}$ of $x$, $\left(xx^{\left(-1\right)}\right)\in\left\{ \ell\right\} _{x}\subseteq\left\{ e\right\} _{\!P}$
and $\left(x^{\left(-1\right)}x\right)\in\left\{ r\right\} _{x}\subseteq\left\{ e\right\} _{\!P}$,
so $\left(x^{\left(-1\right)}\left(xx^{\left(-1\right)}\right)\right)=x^{\left(-1\right)}$
and $\left(\left(x^{\left(-1\right)}x\right)x^{\left(-1\right)}\right)=x^{\left(-1\right)}$.
Thus, $\left(xx^{\left(-1\right)}\right)\in\left\{ \rho\right\} _{x^{-1}}$
and $\left(x^{\left(-1\right)}x\right)\in\left\{ \lambda\right\} _{x^{-1}}$,
so $x^{\left(-1\right)}$ is a preinverse of $x$, and hence a strong
preinverse of $x$.
\end{proof}
Hence, one may alternatively define a groupoid as a poloid $P$ such
that for every $x\in P$ there is a strong pseudoinverse $x^{\left(-1\right)}\in P$
of $x$, and definitions of this form are common in the literature.
\begin{prop}
\label{pro3-1}Let $P$ be a poloid, $x\in P$. Then there is at most
one strong pseudoinverse $x^{\left(-1\right)}\in P$ of $x$.
\end{prop}
\begin{proof}
If $x'$ and $x''$ are strong pseudoinverses of $x$ then $\left(xx'\right)\in\left\{ \ell\right\} _{x}\subseteq\left\{ e\right\} _{\!P}$
and $\left(x''\!x\right)\in\left\{ r\right\} _{x}\subseteq\left\{ e\right\} _{\!P}$,
so $x'=\left(\left(x''x\right)x'\right)=\left(x''\left(xx'\right)\right)=x''$.
\end{proof}
\begin{cor}
Let $P$ be a groupoid. Then every $x\in P$ has a unique strong preinverse
$x^{-1}\in P$.
\end{cor}
\begin{prop}
\label{pro9-1-1}Let $P$ be a groupoid. If $e\in P$ is a two-sided
unit then $e^{-1}=e$.
\end{prop}
\begin{proof}
The assertion is that $\left\{ e\right\} =\boldsymbol{I}\left\{ \right\} _{e}$.
As $\left(ee\right)=e$ we have $\left(\left(ee\right)e\right)=\left(e\left(ee\right)\right)=e$,
so $e\,\mathbf{I}\,e$, so $e\in\mathbf{I}\left\{ \right\} _{e}=\left\{ e^{-1}\right\} $
since preinverses are unique. Thus, $\left\{ e\right\} =\mathbf{I}\left\{ \right\} _{e}$,
so to prove the assertion it suffices to note that $\left(ee^{-1}\right)=\left(e^{-1}e\right)=\left(ee\right)=e$,
so that $\left(ee^{-1}\right)$ and $\left(e^{-1}e\right)$ are two-sided
units. 
\end{proof}
\begin{prop}
\label{pro9}Let $P$ be a groupoid, $x\in P$. Then $\left(x^{-1}\right)^{-1}=x$.
\end{prop}
\begin{proof}
If $x\,\mathit{\boldsymbol{I}}\,x^{-1}$ then $x^{-1}\,\boldsymbol{I}\,x$,
meaning that $x$ is a strong preinverse of $x^{-1}$. Thus, $x\in\boldsymbol{I}\left\{ \right\} _{x^{-1}}=\left\{ \left(x^{-1}\right)^{-1}\right\} $
since strong preinverses are unique.
\end{proof}
\begin{prop}
\label{pro610}Let $P$ be a groupoid, $x,y\in P$. If $\left(xy\right)$
then $\left(xy\right)^{-1}=\left(y^{-1}x^{-1}\right)$.
\end{prop}
\begin{proof}
If $\left(xy\right)$ then $r_{x}=\ell_{x^{-1}}=\ell_{y}=r_{y^{-1}}$,
so $\left(y^{-1}x^{-1}\right)$. Furthermore, we have $\left(x^{-1}x\right)\in\left\{ r_{x}\right\} \subseteq\left\{ e\right\} _{\!P}$
and $\left(yy^{-1}\right)\in\left\{ \ell_{y}\right\} \subseteq\left\{ e\right\} _{\!P}$,
so $\left(xyy^{-1}x^{-1}xy\right)$, and we obtain 
\begin{gather*}
\left(\left(xy\right)\left(y^{-1}x^{-1}\right)\left(xy\right)\right)=\left(x\left(yy^{-1}\right)\left(x^{-1}x\right)y\right)=\left(xy\right).
\end{gather*}
Also, $\left(y^{-1}x^{-1}xyy^{-1}x^{-1}\right)$, and we have
\[
\left(\left(y^{-1}x^{-1}\right)\left(xy\right)\left(y^{-1}x^{-1}\right)\right)=\left(y^{-1}\left(x^{-1}x\right)\left(yy^{-1}\right)x^{-1}\right)=\left(y^{-1}x^{-1}\right).
\]
Thus, $\left(xy\right)\,\mathbf{I}\,\left(y^{-1}x^{-1}\right)$. Also,
\begin{gather*}
\left(\left(xy\right)\left(y^{-1}x^{-1}\right)\right)=\left(x\left(yy^{-1}\right)x^{-1}\right)=\left(xx^{-1}\right)\in\left\{ \ell_{x}\right\} \subseteq\left\{ e\right\} _{\!P},\\
\left(\left(y^{-1}x^{-1}\right)\left(xy\right)\right)=\left(y^{-1}\left(x^{-1}x\right)y\right)=\left(y^{-1}y\right)\in\left\{ r_{y}\right\} \subseteq\left\{ e\right\} _{\!P},
\end{gather*}
 so $\left(xy\right)\,\boldsymbol{I}\,\left(y^{-1}x^{-1}\right)$.
Hence, $\left(y^{-1}x^{-1}\right)\in\boldsymbol{I}\left\{ \right\} _{\left(xy\right)}=\left\{ \left(xy\right)^{-1}\right\} $
since strong preinverses are unique. 
\end{proof}
Since every $x\in P$ has a unique preinverse $x^{-1}$ such that
$\left(x^{-1}\right)^{-1}=x$, there is a unique bijection
\[
\mathfrak{i}:P\rightarrow P,\qquad x\mapsto x^{-1}
\]
 such that, for any $x,y\in P$,
\begin{gather}
\left(\left(x\,\mathfrak{i}\left(x\right)\right)\,x\right)=x,\quad\left(x\,\left(\mathfrak{i}\left(x\right)\,x\right)\right)=x,\nonumber \\
\left(\left(\mathfrak{i}\left(x\right)\,x\right)\,\mathfrak{i}\left(x\right)\right)=\mathfrak{i}\left(x\right),\quad\left(\mathfrak{i}\left(x\right)\left(x\,\mathfrak{i}\left(x\right)\right)\right)=\mathfrak{i}\left(x\right),\label{eq:14}\\
\left(\left(x\,\mathfrak{i}\left(x\right)\right)\,y\right)\rightarrow\left(\left(x\,\mathfrak{i}\left(x\right)\right)\,y\right)=y,\quad\left(y\,\left(x\,\mathfrak{i}\left(x\right)\right)\right)\rightarrow\left(y\,\left(x\,\mathfrak{i}\left(x\right)\right)\right)=y,\nonumber \\
\left(\left(\mathfrak{i}\left(x\right)\,x\right)\,y\right)\rightarrow\left(\left(\mathfrak{i}\left(x\right)\,x\right)\,y\right)=y,\quad\left(y\,\left(\mathfrak{i}\left(x\right)\,x\right)\right)\rightarrow\left(y\,\left(\mathfrak{i}\left(x\right)\,x\right)\right)=y.\nonumber 
\end{gather}

A groupoid $P$ can thus be regarded as an expansion $\left(P,\mathfrak{m},\mathfrak{i},\mathfrak{s},\mathfrak{t}\right)$,
characterized by the identities (\ref{eq:l1}), (\ref{eq:l2}), (\ref{eq:l3}),
and (\ref{eq:14}), of a poloid $\left(P,\mathfrak{m},\mathfrak{s},\mathfrak{t}\right)$,

If $\left(P,\mathfrak{m},\mathfrak{i},\mathfrak{s},\mathfrak{t}\right)$
is a magma then it degenerates to a group $\left(P,\mathfrak{m},\mathfrak{i},1\right)$
where $\mathfrak{s}\left(x\right)=\mathfrak{t}\left(x\right)=1$ for
all $x\in P$.

\subsection{\label{subsec:The-one-sided-skew-poloid}The skew-poloid family}
\begin{defn}
\label{def9}Let $P$ be a left (resp. right) semigroupoid. Then $P$
is 
\begin{enumerate}
\item a \emph{left (resp. right) skew-poloid} when there is a unique twisted
left (resp. right) unit $\varphi_{x}\in P$ (resp. $\psi_{x}\in P$)
for every $x\in P$\emph{;}
\item a \emph{left (resp. right) skew-groupoid} when $P$ is a left (resp.
right) skew-poloid such that for every $x\in P$ there is a strong
left (resp. right) preinverse $x^{-1}\in P$ of $x$ and a unique
$\varphi_{x^{-1}}\in P$ (resp. $\psi_{x^{-1}}\in P$) such that if
$x^{-1}$ is a strong left (resp. right) preinverse of $x$ then $\varphi_{x^{-1}}$
(resp. $\psi_{x^{-1}}$) is the twisted left (resp. right) unit for
$x^{-1}$. 
\end{enumerate}
\end{defn}
In view of the left-right duality in the skew-poloid family, it suffices
to consider only left skew-poloids and left skew-groupoids below.
Note that we cannot in general regard $\left(\left(xy\right)z\right)$
and $\left(x\left(yz\right)\right)$ as equivalent expressions, written
$\left(xyz\right)$, in this case; $\left(x\left(yz\right)\right)$
implies $\left(\left(xy\right)z\right)$, but not conversely.

By Definition \ref{def9}, a left skew-poloid is a left semigroupoid
$P$ such that for every $x\in P$ there is a unique $\varphi_{x}\in P$
such that $\left(\varphi_{x}x\right)=x$ and such that $\left(y\varphi_{x}\right)=y$
for every $y\in P$ such that $\left(y\varphi_{x}\right)$. A left
skew-groupoid is a left skew-poloid $P$ such that for every $x\in P$
there is some $x^{-1}\in P$ such that $\varphi_{x}=\left(xx^{-1}\right)$
and $\varphi{}_{x^{-1}}=\left(x^{-1}x\right)$, so that $\left(\left(xx^{-1}\right)x\right)=x$,
$\left(\left(x^{-1}x\right)x^{-1}\right)=x^{-1}$, $\left(y\left(xx^{-1}\right)\right)=y$
for all $x,y\in P$ such that $\left(y\left(xx^{-1}\right)\right)$,
and $\left(y\left(x^{-1}x\right)\right)=y$ for all $x,y\in P$ such
that $\left(y\left(x^{-1}x\right)\right)$. In addition, because of
the two uniqueness assumptions in Definition \ref{def9}, we have
that if $\left(\left(xx'\right)x\right)\!=\!x,$ $\left(\left(x'x\right)x'\right)\!=\!x'$
and $\left(\left(xx''\right)x\right)=x,\left(\left(x''x\right)x''\right)=x''$
then not only $\left(xx'\right)=\left(xx''\right)$ but also $\left(x'x\right)=\left(x''x\right)$.

A left (or one-sided) skew-poloid is what has been called a \emph{constellation}
\cite{key-6,key-4,key-5}. There is a close relationship between poloids
and (left) skew-poloids, or between categories and constellations,
because both notions formalize the idea of a system of (structured)
sets and many-to-one correspondences between these sets. Without going
into details, the difference between the two notions is that in the
first case many-to-one correspondences are formalized as \emph{functions},
with domains and codomains, whereas in the second case, many-to-one
correspondences are formalized as \emph{prefunctions}, with domains
but without codomains.\footnote{For details, see \cite{key-9,key-5}. In \cite{key-5}, prefunctions
are interpreted as surjective functions; the prefunction $\mathsf{f:}X\rightarrow Y$
is rendered as the function $f:X\rightarrow\mathrm{im}\,f\subseteq Y$.}

\subsubsection*{Left skew-poloids}
\begin{prop}
Let $P$ be a left skew-poloid. If $\varphi,\varphi'\in\left\{ \varphi\right\} _{\!P}$
and $\left(\varphi\varphi'\right)=\left(\varphi'\varphi\right)$ then
$\varphi=\varphi'$.
\end{prop}
\begin{proof}
We have $\varphi=\left(\varphi\varphi'\right)=\left(\varphi'\varphi\right)=\varphi'$.
\end{proof}
\begin{prop}
Let $P$ be a left skew-poloid, $x,y\in P$. Then $\left(xy\right)$
if and only if $\left(x\varphi_{y}\right)$.
\end{prop}
\begin{proof}
If $\left(xy\right)$ then $\left(xy\right)=\left(x\left(\varphi_{y}y\right)\right)=\left(\left(x\varphi_{y}\right)y\right)$,
and if $\left(x\varphi_{y}\right)$ then $\left(x\left(\varphi_{y}y\right)\right)=\left(xy\right)$
since $\left(\varphi_{y}y\right)=y$.
\end{proof}
\begin{cor}
Let $P$ be a left skew-poloid, $x,y\in P$. Then $\left(xy\right)$
if and only if $\left(x\varphi_{y}\right)=x$.
\end{cor}
\begin{prop}
\label{pro5.12}Let $P$ be a left skew-poloid, $x\in P$. Then $\left(\varphi_{x}\varphi_{x}\right)=\varphi_{x}=\varphi_{\varphi_{x}}$.
\end{prop}
\begin{proof}
We have $\varphi_{\varphi_{x}}=\left(\varphi_{\varphi_{x}}\varphi_{x}\right)=\varphi_{x}$.
\end{proof}
\begin{prop}
\label{pro6-1-1}Let $P$ be a left skew-poloid, $x,y\in P$. If $\left(xy\right)$
then $\varphi{}_{\left(xy\right)}=\varphi_{x}$.
\end{prop}
\begin{proof}
If $\left(xy\right)$ then $\left(xy\right)=\left(\left(\varphi_{x}x\right)y\right)=\left(\varphi_{x}\left(xy\right)\right)$,
so $\varphi_{x}\in\left\{ \varphi\right\} _{\left(xy\right)}=\left\{ \varphi_{\left(xy\right)}\right\} $.
\end{proof}
In view of Definition \ref{def9}, every left skew-poloid $P$ can
be equipped with a unique surjective function 
\begin{gather*}
\mathfrak{s}:P\rightarrow\left\{ \varphi\right\} _{\!P},\qquad x\mapsto\varphi{}_{x}
\end{gather*}
 such that, for any $x,y\in P$,
\begin{equation}
\left(\mathfrak{s}\left(x\right)\,x\right)=x,\quad\left(y\,\mathfrak{s}\left(x\right)\right)\rightarrow\left(y\,\mathfrak{s}\left(x\right)\right)=y.\label{eq:ll1}
\end{equation}

By Proposition \ref{pro5.12}, $\mathfrak{s}\left(\varphi{}_{x}\right)=\varphi{}_{x}$
for all $\varphi{}_{x}\in\left\{ \varphi\right\} _{\!P}$.

A left skew-poloid can thus be regarded as an expansion $\left(P,\mathfrak{m},\mathfrak{s}\right)$,
characterized by the identities (\ref{eq:l1}) and (\ref{eq:ll1}),
of a left semigroupoid $\left(P,\mathfrak{m}\right)$.

If $\left(P,\mathfrak{m},\mathfrak{s}\right)$ is a magma then it
degenerates to a monoid $\left(P,\mathfrak{m},1\right)$ where $\mathfrak{s}\left(x\right)=\mathfrak{t}\left(x\right)=1$
for all $x\in P$.\footnote{If $\left(\left(xy\right)z\right)$ then $\left(x\left(yz\right)\right)$
since $\left(yz\right)$, so if $\left(\left(xy\right)z\right)$ or
$\left(x\left(yz\right)\right)$ then $\left(\left(xy\right)z\right)=\left(x\left(yz\right)\right)$.}

\subsubsection*{Left skew-groupoids }
\begin{prop}
Let $P$ be a left skew-groupoid, $x\in P$. Then there is at most
one strong right preinverse $x^{-1}\in P$ of $x$.
\end{prop}
\begin{proof}
Let $x'$ and $x''$ be strong right preinverses of $x$. By the uniqueness
of $\varphi_{x}$, $\varphi_{x'}$, and $\varphi_{x''}$ for all $x\in P$,
we have $\varphi_{x}=\left(xx'\right)=\left(xx''\right)$ as well
as $\varphi_{x'}=\left(x'x\right)=\left(x''x\right)=\varphi_{x''}$.
Thus, $x'\!=\!\left(\left(x'x\right)x'\right)\!=\!\left(\left(x''x\right)x'\right)=\!\left(x''\left(xx'\right)\right)\!=\!\left(x''\left(xx''\right)\right)\!=\!\left(\left(x''x\right)x''\right)\!=\!x''$.
\end{proof}
\begin{cor}
Let $P$ be a left skew-groupoid, $x\in P$. Then $x$ has a unique
strong right preinverse $x^{-1}\in P$.
\end{cor}
\begin{prop}
Let $P$ be a left skew-groupoid. If $\varphi\in\left\{ \varphi\right\} _{\!P}$
then $\varphi^{-1}=\varphi$. 
\end{prop}
\begin{proof}
The assertion is that $\left\{ \varphi\right\} =\boldsymbol{I}^{+}\!\left\{ \right\} _{\varphi}$.
By Proposition \ref{pro5.12}, $\left(\varphi\varphi\right)=\varphi$,
so $\left(\left(\varphi\varphi\right)\varphi\right)=\varphi$, so
$\varphi\,\mathbf{I}^{+}\varphi$, so $\varphi\in\mathbf{I}^{+}\left\{ \right\} _{\varphi}=\left\{ \varphi^{-1}\right\} $
since strong right preinverses are unique. Thus, $\left\{ \varphi\right\} =\mathbf{I}^{+}\left\{ \right\} _{\varphi}$,
so to prove the assertion it suffices to note that $\left(\varphi\varphi^{-1}\right)=\left(\varphi^{-1}\varphi\right)=\left(\varphi\varphi\right)=\varphi$,
so that $\left(\varphi\varphi^{-1}\right),\left(\varphi^{-1}\varphi\right)\in\left\{ \varepsilon\right\} _{\!P}$. 
\end{proof}
\begin{prop}
Let $P$ be a left skew-groupoid, $x\in P$. Then $\left(x^{-1}\right)^{-1}=x$.
\end{prop}
\begin{proof}
If $x\,\boldsymbol{I}{}^{+}x^{-1}$ then $x^{-1}\,\boldsymbol{I}{}^{+}x$,
meaning that $x$ is a strong right preinverse of $x^{-1}$. Thus,
$x\in\boldsymbol{I}^{+}\!\left\{ \right\} _{x^{-1}}=\left\{ \left(x^{-1}\right)^{-1}\right\} $
since strong right preinverses are unique. 
\end{proof}
Since every $x\in P$ has a unique preinverse $x^{-1}$ such that
$\left(x^{-1}\right)^{-1}=x$, there is a unique bijection
\[
\mathfrak{i}:P\rightarrow P,\qquad x\mapsto x^{-1}
\]
such that, for all $x,y\in P$,
\begin{gather}
\left(\left(x\,\mathfrak{i}\left(x\right)\right)\,x\right)=x,\quad\left(\left(\mathfrak{i}\left(x\right)\,x\right)\,\mathfrak{i}\left(x\right)\right)=\mathfrak{i}\left(x\right),\label{eq:l12}\\
\left(y\,\left(x\,\mathfrak{i}\left(x\right)\right)\right)\rightarrow\left(y\,\left(x\,\mathfrak{i}\left(x\right)\right)\right)=y,\quad\left(y\,\left(\mathfrak{i}\left(x\right)\,x\right)\right)\rightarrow\left(y\,\left(\mathfrak{i}\left(x\right)\,x\right)\right)=y.\nonumber 
\end{gather}

A left skew-groupoid can be thus be regarded as an expansion $\left(P,\mathfrak{m},\mathfrak{i},\mathfrak{s}\right)$,
characterized by the identities (\ref{eq:l1}), (\ref{eq:ll1}), and
(\ref{eq:l12}), of a left skew-poloid $\left(P,\mathfrak{m},\mathfrak{s}\right)$.

If $\left(P,\mathfrak{m},\mathfrak{i},\mathfrak{s}\right)$ is a magma
then it degenerates to a group $\left(P,\mathfrak{m},\mathfrak{i},1\right)$
where $\mathfrak{s}\left(x\right)=1$ for all $x\in P$.

\section{\label{sec:Prepoloids-and-pregroupoids}Prepoloids and pregroupoids
with restricted binary operations}

In a poloid, $\left(xy\right)$ if and only if $r_{x}=\ell_{y}$.
In a prepoloid, $\rho_{x}=\lambda_{y}$ implies $\left(xy\right)$
by Proposition \ref{pro13}, but $\rho_{x}=\lambda_{y}$ is not a
necessary condition for $\left(xy\right)$. If we retain only those
products $\left(xy\right)$ for which $\rho_{x}=\lambda_{y}$, we
obtain a magmoid $P\left[\boldsymbol{\mathsf{m}}\right]$ with the
same elements as $P$ but restricted multiplication $\boldsymbol{\mathsf{m}}:\left(x,y\right)\mapsto\left(x\cdot y\right)$.
By definition, we then have $\left(x\cdot y\right)$ if and only if
$\rho_{x}=\lambda_{y}$, as in a poloid, and under suitable conditions
the restricted magmoid does indeed become a poloid.

By similarly restricting the binary operation, we can derive a groupoid
from a pregroupoid, a skew-poloid from a skew-prepoloid and a skew-groupoid
from a skew-pregroupoid. 

\subsection{\label{subsec:From-prepoloids-to}From prepoloids to poloids }
\begin{defn}
\label{def6}Let $P$ be\emph{ }a prepoloid with binary operation
$\mathfrak{m}:\left(x,y\right)\mapsto\left(xy\right)$, $x,y\in P$.
The\emph{ restricted binary operation} on the carrier set of $P$
is a binary operation $\boldsymbol{\mathsf{m}}:\left(x,y\right)\mapsto\left(x\cdot y\right)$
such that if $\left(x\cdot y\right)$ and $\left(xy\right)$ then
$\left(x\cdot y\right)=\left(xy\right)$, and $\left(x\cdot y\right)$
if and only if relative to $\mathfrak{m}$ there is some $\rho_{x}\in\left\{ \rho\right\} _{x}$
and some $\lambda_{y}\in\left\{ \lambda\right\} _{y}$ such that $\rho_{x}=\lambda_{y}$.
\end{defn}
In particular, Definition \ref{def6} applies to magmoids where $\lambda_{x}$
and $\rho_{x}$ are unique local units for $x\in P$.

Note that if $\left(x\cdot y\right)$ then $\rho_{x}=\lambda_{y}$,
so $\left(xy\right)$, so if $\left(x\cdot y\right)$ then $\left(x\cdot y\right)=\left(xy\right)$
by Definition \ref{def6}.
\begin{lem}
\label{pro81}Let $P$ be\emph{ }a prepoloid with unique local units.
If $\lambda{}_{x}=\lambda{}_{\lambda_{x}}=\rho_{\lambda_{x}}$ and
$\rho_{x}=\lambda_{\rho_{x}}=\rho_{\rho_{x}}$ for all $x\in P$ and
also $\lambda_{\left(xy\right)}=\lambda_{x}$ and $\rho_{\left(xy\right)}=\rho_{y}$
for all $x,y\in P$ such that $\left(xy\right)$ then $P\left[\boldsymbol{\mathsf{m}}\right]$
is a poloid where $\ell_{x}=\lambda_{x}$ and $r_{x}=\rho_{x}$ for
all $x\in P\left[\boldsymbol{\mathsf{m}}\right]$.
\end{lem}
\begin{proof}
We first prove that $P\left[\boldsymbol{\mathsf{m}}\right]$ is a
semigroupoid. If $\left(x\cdot\left(y\cdot z\right)\right)$ then
$\rho_{x}=\lambda_{\left(y\cdot z\right)}$ and $\rho_{y}=\lambda_{z}$,
so $\rho_{x}=\lambda_{\left(y\cdot z\right)}=\lambda_{\left(yz\right)}=\lambda_{y}$.
Thus, $\rho_{\left(x\cdot y\right)}=\rho_{\left(xy\right)}=\rho_{y}=\lambda_{z}$,
so $\left(\left(x\cdot y\right)\cdot z\right)$, so $\left(x\cdot\left(y\cdot z\right)\right)=\left(x\left(yz\right)\right)=\left(\left(xy\right)z\right)=\left(\left(x\cdot y\right)\cdot z\right)$. 

Dually, if $\left(\left(x\cdot y\right)\cdot z\right)$ then $\rho_{\left(x\cdot y\right)}=\lambda_{z}$
and $\rho_{x}=\lambda_{y}$, so $\rho_{y}=\rho_{\left(xy\right)}=\rho_{\left(x\cdot y\right)}=\lambda_{z}$.
Thus, $\rho_{x}=\lambda_{y}=\lambda_{\left(yz\right)}=\lambda_{\left(y\cdot z\right)}$,
so $\left(x\cdot\left(y\cdot z\right)\right)$, so $\left(\left(x\cdot y\right)\cdot z\right)=\left(\left(xy\right)z\right)=\left(x\left(yz\right)\right)=\left(x\cdot\left(y\cdot z\right)\right)$. 

Also, if $\left(x\cdot y\right)$ and $\left(y\cdot z\right)$ then
$\rho_{x}=\lambda_{y}$ and $\rho_{y}=\lambda_{z}$, so $\rho_{x}=\lambda_{y}=\lambda_{\left(yz\right)}=\lambda_{\left(y\cdot z\right)}$
and $\rho_{\left(x\cdot y\right)}=\rho_{\left(xy\right)}=\rho_{y}=\lambda_{z}$.
Hence, $\left(x\cdot\left(y\cdot z\right)\right)$ and $\left(\left(x\cdot y\right)\cdot z\right)$,
so $\left(x\cdot\left(y\cdot z\right)\right)=\left(x\left(yz\right)\right)=\left(\left(xy\right)z\right)=\left(\left(x\cdot y\right)\cdot z\right)$.

It remains to show that $\lambda_{x}$ and $\rho_{x}$ are two-sided
units such that $\left(\lambda{}_{x}\cdot x\right)=x$ and $\left(x\cdot\rho{}_{x}\right)=x$
for all $x\in P\left[\boldsymbol{\mathsf{m}}\right]$. If $\left(\lambda_{x}\cdot y\right)$
then $\lambda{}_{x}=\rho_{\lambda_{x}}=\lambda{}_{y}$, so $\left(\lambda_{x}\cdot y\right)=\left(\lambda_{x}y\right)=\left(\lambda_{y}y\right)=y$,
and if $\left(y\cdot\lambda{}_{x}\right)$ then $\rho_{y}=\lambda{}_{\lambda_{x}}=\lambda{}_{x}$,
so $\left(y\cdot\lambda{}_{x}\right)=\left(y\lambda{}_{x}\right)=\left(y\rho_{y}\right)=y$. 

Similarly, if $\left(\rho_{x}\cdot y\right)$ then $\rho_{x}=\rho_{\rho_{x}}=\lambda_{y}$,
so $\left(\rho_{x}\cdot y\right)=\left(\rho_{x}y\right)=\left(\lambda_{y}y\right)=y$,
and if $\left(y\cdot\rho_{x}\right)$ then $\rho_{y}=\lambda_{\rho_{x}}=\rho_{x}$
so $\left(y\cdot\rho_{x}\right)=\left(y\rho_{x}\right)=\left(y\rho_{y}\right)=y$.

Thus, $\lambda_{x}$ and $\rho_{x}$ are two-sided units in $P\left[\boldsymbol{\mathsf{m}}\right]$
for all $x\in P\left[\boldsymbol{\mathsf{m}}\right]$, and we have
$\left(\lambda_{x}\cdot x\right)=\left(\lambda_{x}x\right)=x$ and
$\left(x\cdot\rho{}_{x}\right)=\left(x\rho{}_{x}\right)=x$ for all
$x\in P\left[\boldsymbol{\mathsf{m}}\right]$. 
\end{proof}
Combining Lemma \ref{pro81} with Propositions \ref{pro11} and \ref{pro12}
we obtain the following results:
\begin{thm}
If $P$ is\emph{ }a prepoloid with unique local units then $P\left[\boldsymbol{\mathsf{m}}\right]$
is a poloid.
\end{thm}
\begin{cor}
If $P$ is a bi-unital semigroup with unique local units then $P\left[\boldsymbol{\mathsf{m}}\right]$
is a poloid.\pagebreak{}
\end{cor}

\subsection{\label{subsec:From-pregroupoids-to}From pregroupoids to groupoids}
\begin{defn}
Let $P$ be\emph{ }a pregroupoid with binary operation $\mathfrak{m}:\left(x,y\right)\mapsto\left(xy\right)$,
$x,y\in P$. The\emph{ restricted binary operation }on the carrier
set of $P$ is a binary operation $\boldsymbol{\boldsymbol{\mathsf{m}}}:\left(x,y\right)\mapsto\left(x\cdot y\right)$
such that if $\left(x\cdot y\right)$ and $\left(xy\right)$ then
$\left(x\cdot y\right)=\left(xy\right)$, and $\left(x\cdot y\right)$
if and only if there is a canonical local right unit $\boldsymbol{\uprho}_{x}=\left(x^{-1}x\right)$
for $x$ and a canonical local left unit $\boldsymbol{\uplambda}_{y}=\left(yy^{-1}\right)$
for $y$ such that $\boldsymbol{\uprho}_{x}=\boldsymbol{\uplambda}_{y}$.
\end{defn}
If $\left(x\cdot y\right)$ then there are $\boldsymbol{\uprho}_{x},\boldsymbol{\uplambda}_{y}$
such that $\boldsymbol{\uprho}_{x}=\boldsymbol{\uplambda}_{y}$, so
$\left(xy\right)$ by Proposition \ref{pro13}, so $\left(x\cdot y\right)=\left(xy\right)$. 
\begin{lem}
\label{pro82}Let $P$ be a pregroupoid with unique canonical local
units. If $\boldsymbol{\uplambda}{}_{x}=\boldsymbol{\uplambda}{}_{\boldsymbol{\uplambda}_{x}}=\boldsymbol{\uprho}_{\lambda_{x}}$
and $\boldsymbol{\uprho}_{x}=\boldsymbol{\uplambda}_{\boldsymbol{\uprho}_{x}}=\boldsymbol{\uprho}_{\boldsymbol{\uprho}_{x}}$
for all $x\in P$ and also $\boldsymbol{\uplambda}_{\left(xy\right)}=\boldsymbol{\uplambda}_{x}$
and $\boldsymbol{\uprho}_{\left(xy\right)}=\boldsymbol{\uprho}_{y}$
for all $x,y\in P$ such that $\left(xy\right)$ then $P\left[\boldsymbol{\mathsf{m}}\right]$
is a groupoid where $\ell_{x}=\boldsymbol{\uplambda}{}_{x}$ and $r_{x}=\boldsymbol{\uprho}_{x}$
for all $x\in P\left[\boldsymbol{\mathsf{m}}\right]$.
\end{lem}
\begin{proof}
We can prove that $P\left[\boldsymbol{\mathsf{m}}\right]$ is a poloid
where $\ell_{x}=\boldsymbol{\uplambda}{}_{x}$ and $r_{x}=\boldsymbol{\uprho}_{x}$
for all $x\in P\left[\boldsymbol{\mathsf{m}}\right]$ by using the
argument in the proof of Lemma \ref{pro81} again. It remains to show
that every $x\in P\left[\boldsymbol{\mathsf{m}}\right]$ has a strong
preinverse in $P\left[\boldsymbol{\mathsf{m}}\right]$. Let $x^{-1}$
be a preinverse of $x$ in $P$. As $\boldsymbol{\uprho}_{x}=\left(x^{-1}x\right)=\boldsymbol{\lambda}_{x^{-1}}$
we have $\left(x\cdot x^{-1}\right)$, and as $\boldsymbol{\uprho}_{x^{-1}}=\left(xx^{-1}\right)=\boldsymbol{\lambda}_{x}$
we have $\left(x^{-1}\cdot x\right)$, so $\left(\left(x\cdot x^{-1}\right)\cdot x\right)$
and $\left(x\cdot\left(x^{-1}\cdot x\right)\right)$. Thus, $\left(\left(x\cdot x^{-1}\right)\cdot x\right)=\left(\left(xx^{-1}\right)x\right)=x$
and $\left(x\cdot\left(x^{-1}\cdot x\right)\right)=\left(x\left(x^{-1}x\right)\right)=x$,
so $\left(x\cdot x^{-1}\right)=\ell_{x}\in\left\{ e\right\} _{\!P}$
and $\left(x^{-1}\cdot x\right)=r_{x}\in\left\{ e\right\} _{\!P}$
by the uniqueness of effective units in $P\left[\boldsymbol{\mathsf{m}}\right]$.
Thus $x^{-1}$ is a strong pseudoinverse of $x$ in $P\left[\boldsymbol{\mathsf{m}}\right]$
and hence a strong preinverse of $x$ in $P\left[\boldsymbol{\mathsf{m}}\right]$
by Proposition \ref{pro9-1}.
\end{proof}
Combining Lemma \ref{pro82} with Propositions \ref{pro79} and \ref{pro710},
we obtain the following results:
\begin{thm}
If $P$ is a pregroupoid with unique canonical local units then $P\left[\boldsymbol{\mathfrak{m}}\right]$
is a groupoid.
\end{thm}
\begin{cor}
\label{cor82}If $S$ is a regular semigroup with unique canonical
local units then $S\left[\boldsymbol{\mathsf{m}}\right]$ is a groupoid.
\end{cor}
\begin{thm}
If $P$ is a pregroupoid with unique preinverses then $P\left[\boldsymbol{\mathsf{m}}\right]$
is a groupoid.
\end{thm}
\begin{proof}
If preinverses are unique then $\left(xx^{-1}\right)$ and $\left(x^{-1}x\right)$
are uniquely determined by $x$, so canonical local units are unique.
\end{proof}
\begin{cor}
\label{xor83}If $S$ is an inverse semigroup then $S\left[\boldsymbol{\mathsf{m}}\right]$
is a groupoid.
\end{cor}
It is clear that Corollaries \ref{cor82} and \ref{xor83} are related
to the so-called Ehresmann-Schein-Nampooribad theorem in semigroup
theory: groupoids correspond to inverse semigroups and, more generally,
to regular semigroups whose canonical local units are unique.

\subsection{\label{subsec:From-left-skew-prepoloids}From left skew-prepoloids
to left skew-poloids }
\begin{defn}
Let $P$ be\emph{ }a left skew-prepoloid with a binary operation\linebreak{}
 $\mathfrak{m}:\left(x,y\right)\mapsto\left(xy\right)$, $x,y\in P$.
The\emph{ restricted multiplication} on the carrier set of $P$ is
a binary operation $\boldsymbol{\mathsf{m}}:\left(x,y\right)\mapsto\left(x\cdot y\right)$
such that if $\left(x\cdot y\right)$ and $\left(xy\right)$ then
$\left(x\cdot y\right)=\left(xy\right)$, and $\left(x\cdot y\right)$
if and only if there is some local left unit $\lambda_{y}$ for $y$
such that $\left(x\lambda_{y}\right)=x$.
\end{defn}
Note that if $\left(x\cdot y\right)$ then $\left(xy\right)$ by Proposition
\ref{pro711}, so then $\left(x\cdot y\right)=\left(xy\right)$.
\begin{lem}
\label{pro83}Let $P$ be a left skew-prepoloid with unique local
left units. If $\left(\lambda_{x}\lambda_{x}\right)=\lambda_{x}=\lambda{}_{\lambda_{x}}$
for all $x\in P$ and $\lambda_{\left(xy\right)}=\lambda_{x}$ for
all $x,y\in P$ such that $\left(xy\right)$ then $P\left[\boldsymbol{\mathsf{m}}\right]$
is a left skew-poloid where $\varphi_{x}=\lambda_{x}$ for all $x\in P\left[\boldsymbol{\mathsf{m}}\right]$.
\end{lem}
\begin{proof}
If $\left(x\cdot\left(y\cdot z\right)\right)$ then $\left(y\lambda_{z}\right)=y$,
so $x=\left(x\lambda_{\left(y\cdot z\right)}\right)=\left(x\lambda_{\left(yz\right)}\right)=\left(x\lambda_{y}\right)$,
so $\left(x\cdot y\right)$, so $\left(x\cdot y\right)=\left(xy\right)$.
Thus, $\left(xy\right)=\left(x\left(y\lambda_{z}\right)\right)=\left(\left(xy\right)\lambda_{z}\right)$,
so $\left(\left(xy\right)\cdot z\right)=\left(\left(x\cdot y\right)\cdot z\right)$.
Hence, if $\left(\left(x\cdot\left(y\cdot z\right)\right)\right)$
then $\left(x\cdot\left(y\cdot z\right)\right)=\left(x\left(yz\right)\right)=\left(\left(xy\right)z\right)=\left(\left(x\cdot y\right)\cdot z\right)$.

If $\left(x\cdot y\right)$ and $\left(y\cdot z\right)$ then $\left(x\lambda_{y}\right)=x$
and $\left(y\lambda_{z}\right)=y$. Hence, $x=\left(x\lambda_{y}\right)=\left(x\lambda_{\left(yz\right)}\right)=\left(x\lambda_{\left(y\cdot z\right)}\right)$,
so $\left(x\cdot\left(y\cdot z\right)\right)$, so again $\left(x\cdot\left(y\cdot z\right)\right)=\left(x\left(yz\right)\right)=\left(\left(xy\right)z\right)=\left(\left(x\cdot y\right)\cdot z\right)$.

If $\left(y\cdot\lambda_{x}\right)$ then $\left(y\lambda_{\lambda_{x}}\right)=y$
so $\left(y\cdot\lambda_{x}\right)=\left(y\lambda_{x}\right)=y$ since
$\lambda_{x}=\lambda_{\lambda_{x}}$. We also have $\left(\lambda_{x}\lambda_{x}\right)=\lambda_{x}$,
so $\left(\lambda_{x}\cdot x\right)$, so $\left(\lambda_{x}\cdot x\right)=\left(\lambda_{x}x\right)=x$.
Thus, $\lambda_{x}$ is a right unit and a local left unit for $x$,
that is, a twisted left unit $\varphi_{x}$ for $x$. The uniqueness
of $\varphi_{x}$ follows from the uniqueness of $\lambda_{x}$.
\end{proof}
Combining Lemma \ref{pro83} with Propositions \ref{pro714} and \ref{pro715},
we obtain the following results:
\begin{thm}
If $P$ is a left skew-prepoloid with unique local left units then
$P\left[\boldsymbol{\mathsf{m}}\right]$ is a left skew-poloid.
\end{thm}
\begin{cor}
If $S$ is a left unital semigroup with unique local left units then
$S\left[\boldsymbol{\mathsf{m}}\right]$ is a left skew-poloid.
\end{cor}
The fact that a left skew-poloid can be constructed from a left skew-prepoloid
is related to the fact that an inductive constellation can be constructed
from a left restriction semigroup \cite{key-6,key-4}.

\subsection{\label{subsec:From-left-skew-pregroupoids}From left skew-pregroupoids
to left skew-groupoids}
\begin{defn}
Let $P$ be\emph{ }a left skew-pregroupoid with binary operation\linebreak{}
 $\mathfrak{m}:\left(x,y\right)\mapsto\left(xy\right)$, $x,y\in P$.
The\emph{ restricted binary operation} on the carrier set of $P$
is a binary operation $\boldsymbol{\mathsf{m}}:\left(x,y\right)\mapsto\left(x\cdot y\right)$
such that if $\left(x\cdot y\right)$ and $\left(xy\right)$ then
$\left(x\cdot y\right)=\left(xy\right)$, and $\left(x\cdot y\right)$
if and only if there is a canonical local left unit $\boldsymbol{\uplambda}_{y}=\left(yy^{-1}\right)$
such that $\left(x\boldsymbol{\uplambda}_{y}\right)=x$ .
\end{defn}
If $\left(x\cdot y\right)$ then $\left(x\boldsymbol{\uplambda}_{y}\right)$
so $\left(xy\right)$ by Proposition \ref{pro711}, so $\left(x\cdot y\right)=\left(xy\right)$.
\begin{lem}
\label{pro85}Let $P$ be a left skew-pregroupoid with unique canonical
local left units. If $\left(\boldsymbol{\uplambda}_{x}\boldsymbol{\uplambda}_{x}\right)=\boldsymbol{\uplambda}_{x}=\boldsymbol{\uplambda}_{\boldsymbol{\uplambda}_{x}}$
for all $x\in P$ and $\boldsymbol{\uplambda}_{\left(xy\right)}=\boldsymbol{\uplambda}_{x}$
for all $x,y\in P$ such that $\left(xy\right)$ then $P\left[\boldsymbol{\mathsf{m}}\right]$
is a left skew-groupoid where $\varphi_{x}=\boldsymbol{\uplambda}_{x}$
and $\varphi_{x^{-1}}=\boldsymbol{\uplambda}_{x^{-1}}$ for all $x\in P\left[\boldsymbol{\mathsf{m}}\right]$.
\end{lem}
\begin{proof}
It can be proved that $P$ is a left skew-poloid with unique twisted
left units $\varphi_{x}=\boldsymbol{\uplambda}_{x}$ by using the
argument in the proof of Lemma \ref{pro83} again. To complete the
proof, we first show that for every $x\in P\left[\boldsymbol{\mathsf{m}}\right]$
there is a corresponding right preinverse $x^{-1}\in P\left[\boldsymbol{\mathsf{m}}\right]$.

If $x^{-1}\in P$ is a right preinverse of $x$ then $\left(xx^{-1}\right)=\boldsymbol{\uplambda}_{x}$
and $\left(x^{-1}x\right)=\boldsymbol{\uplambda}_{x^{-1}}$. Thus,
\begin{gather*}
\left(x\boldsymbol{\uplambda}_{x^{-1}}\right)=\left(x\left(x^{-1}x\right)\right)=\left(\left(xx^{-1}\right)x\right)=\left(\boldsymbol{\uplambda}_{x}x\right)=x,\\
\left(x^{-1}\boldsymbol{\uplambda}_{x}\right)=\left(x^{-1}\left(xx^{-1}\right)\right)=\left(\left(x^{-1}x\right)x^{-1}\right)=\left(\boldsymbol{\uplambda}_{x^{-1}}x^{-1}\right)=x^{-1},
\end{gather*}
so $\left(x\cdot x^{-1}\right)$ and $\left(x^{-1}\cdot x\right)$
since $\boldsymbol{\uplambda}_{x^{-1}}\!=\!\boldsymbol{\uplambda}_{\left(x^{-1}\right)}$
by the uniqueness of canonical local left units (Section \ref{subsec:Canonical}).
Hence, $\left(\left(x\cdot x^{-1}\right)\cdot x\right)$ and $\left(\left(x^{-1}\cdot x\right)\cdot x^{-1}\right)$,
so
\[
\left(\left(x\cdot x^{-1}\right)\cdot x\right)=\left(\left(xx^{-1}\right)x\right)=x,\quad\left(\left(x^{-1}\cdot x\right)\cdot x^{-1}\right)=\left(\left(x^{-1}x\right)x^{-1}\right)=x^{-1},
\]
 so $x^{-1}$ is a right preinverse of $x$ in $P\left[\boldsymbol{\mathsf{m}}\right]$. 

It is shown in the proof of Lemma \ref{pro83} that if $\left(y\cdot\lambda_{x}\right)$
then $\left(y\cdot\lambda_{x}\right)=y$ for any $x,y\in P$. It can
be shown by a similar argument that if $\left(y\cdot\boldsymbol{\uplambda}_{x}\right)$
then $\left(y\cdot\boldsymbol{\uplambda}_{x}\right)=y$, and that
if $\left(y\cdot\boldsymbol{\uplambda}_{x^{-1}}\right)$ then $\left(y\cdot\boldsymbol{\uplambda}_{x^{-1}}\right)=\left(y\cdot\boldsymbol{\uplambda}_{\left(x^{-1}\right)}\right)=y$.
Hence, $x^{-1}$ is a strong right preinverse of $x$ in $P\left[\boldsymbol{\mathsf{m}}\right]$
with $\varphi_{x}=\boldsymbol{\uplambda}_{x}$ and $\varphi_{x^{-1}}=\boldsymbol{\uplambda}_{x^{-1}}=\boldsymbol{\uplambda}_{\left(x^{-1}\right)}$.

The uniqueness of $\varphi_{x^{-1}}$ follows from the uniqueness
of $\boldsymbol{\uplambda}_{\left(x^{-1}\right)}$.
\end{proof}
Combining Lemma \ref{pro85} with Propositions \ref{pro720} and \ref{pro721-1},
we obtain the following results:
\begin{thm}
If $P$ is a left skew-pregroupoid with unique canonical local left
units, then $P\left[\boldsymbol{\mathsf{m}}\right]$ is a left skew-groupoid.
\end{thm}
\begin{thm}
\label{the86}If $P$ is a left skew-pregroupoid with unique preinverses
then $P\left[\boldsymbol{\mathsf{m}}\right]$ is a left skew-groupoid.
\end{thm}
\begin{proof}
If preinverses are unique then $\left(xx^{-1}\right)$ is determined
by $x$, so canonical local left units are unique.
\end{proof}

\subsection{Final remarks}

Sections \ref{sec:Prepoloids-and-related} and \ref{sec:Poloids-and-related}
are structured around three main distinctions, exemplified by the
distinctions between poloids and skew-poloids, poloids and prepoloids,
and poloids and groupoids. These distinctions generate eight types
of poloid-like magmoids: prepoloids, pregroupoids, skew-prepoloids,
skew-pregroupoids, poloids, groupoids, skew-poloids, and skew-groupoids.

The most superficial of the three main distinctions is perhaps the
one between poloid-like and skew-poloid-like magmoids. The equivalence
of these two notions was briefly discussed in Section \ref{subsec:The-one-sided-skew-poloid}.
The idea that the difference between them ultimately reflects the
way mappings are formalized \textendash{} as functions with domains
and codomains, corresponding to the two-sided objects, or as prefunctions
with only domains, corresponding to the one-sided objects \textendash{}
is implicit in \cite{key-9}.\footnote{In particular, a transformation magmoid is associative, whereas a
pretransformation magmoid is just skew-associative; see Facts 3 and
4 in \cite{key-9}.} The close connection between the two notions does not mean, though,
that the corresponding distinction is trivial or of little interest;
showing the material equivalence of superficially different notions
is an important accomplishment in mathematics. It seems that much
work remains to be done in this connection.

Next, the distinction between poloid-like and prepoloid-like magmoids
is also a distinction between two closely related notions. We have
seen in Section \ref{sec:Prepoloids-and-pregroupoids} that (skew-)prepoloids
with unique local units and (skew-)pregroupoids with unique preinverses
can be transformed into corresponding (skew-)poloids and (skew-)group\-oids.
Conversely, it is shown in the literature that prepoloid-like magmoids
can be recovered from corresponding poloid-like magmoids with additional
structure. The relationship between prepoloid-like and poloid-like
magmoids, in particular semigroups and poloids/groupoids, has been
researched extensively, but not exhaustively, starting with Ehresmann.

Finally, there is the distinction between different kinds of poloid-like
magmoids exemplified by that between monoids and groups. This is obviously
a significant distinction, though one which hardly needs further comment
here.

\appendix

\section{Heap-like algebras with partial operations}

As we know, a group, with a binary operation $\left(x,y\right)\mapsto xy$,
is closely related to a corresponding heap with a ternary operation,
written $\left(x,y,z\right)\mapsto\left[x,y,z\right]$ or $\left(x,y,z\right)\mapsto\left[xyz\right]$.
Analogously, an involution magmoid, with a partial binary operation
$\left(x,y\right)\mapsto xy$ and total unary operation $x\mapsto x^{*}$,
is closely related to a corresponding algebra with a partial ternary
operation $\left(x,y,z\right)\mapsto\left[xyz\right]$ as described
below.

\subsubsection*{Semiheapoids and semiheaps from involution semigroupoids}

One can use any total involution $*$ on a semigroupoid $P$ to define
a partial ternary operation 
\[
\mathsf{t}:P\times P\times P\rightarrow P,\qquad\left(x,y,z\right)\mapsto\left[xyz\right]:=\left(xy^{*}z\right)
\]
on $P$, with $\left[xyz\right]$ being defined if and only if $\left(xy^{*}z\right)$. 

Let $\left[\left[xyz\right]uv\right]$, $\left[xy\left[zuv\right]\right]$
and $\left[x\left[uzy\right]v\right]$ be defined. Then $\left[\left[xyz\right]uv\right]=\left(\left(xy^{*}z\right)u^{*}v\right)=\left(xy^{*}zu^{*}v\right)$,
$\left[xy\left[zuv\right]\right]=\left(xy^{*}\left(zu^{*}v\right)\right)=\left(xy^{*}zu^{*}v\right)$
and $\left[x\left[uzy\right]v\right]=\left(x\left(uz^{*}y\right)^{*}v\right)=\left(x\left(y^{*}\left(z^{*}\right)^{*}u^{*}\right)v\right)=\left(xy^{*}zu^{*}v\right).$
Thus,
\begin{equation}
\left[\left[xyz\right]uv\right]=\left[x\left[uzy\right]v\right]=\left[xy\left[zuv\right]\right].\label{eq:lh1}
\end{equation}
Extending Wagner's terminology \cite{key-18}, we call any set $P$
with a partial ternary operation $\mathsf{t}:\left(x,y,z\right)\mapsto\left[xyz\right]$
satisfying (\ref{eq:lh1}) when all terms are defined a \emph{semiheapoid}.
If $\mathsf{t}$ is a total function satisfying (\ref{eq:lh1}) then
$P$ is a \emph{semiheap}.

\subsubsection*{Semiheapoids and semiheaps from groupoids; heapoids and heaps}

By Propositions \ref{pro9} and \ref{pro610}, every groupoid has
a total involution, namely the function $x\mapsto x^{-1}$. In this
case, the identities (\ref{eq:14}) give rise to identities applying
to $\mathsf{t}$. Specifically, if $\left(xx^{-1}y\right)$ then $\left(xx^{-1}y\right)=\left(\ell_{x}y\right)=y$
and if $\left(yx^{-1}x\right)$ then $\left(yx^{-1}x\right)=\left(yr_{x}\right)=y$,
so 
\begin{equation}
\left[xxy\right]=\left[yxx\right]=y.\label{eq:lh2}
\end{equation}

We call a non-empty set $P$ with a partial ternary operation $\mathsf{t}$
satisfying (\ref{eq:lh1}) and (\ref{eq:lh2}) when all terms are
defined a \emph{heapoid}. If $\mathsf{t}$ is a total function satisfying
(\ref{eq:lh1}) and (\ref{eq:lh2}) then $P$ is a \emph{heap}.

\subsubsection*{Semiheapoids and semiheaps from pregroupoids; preheapoids and preheaps}

By Pro\-positions \ref{pro77} and \ref{pro78}, every pregroupoid
$P$ with unique preinverses has a total involution, namely the function
$x\mapsto x^{-1}$. If the involution has this form then 
\[
\left(yy^{-1}zz^{-1}x\right)=\left(zz^{-1}yy^{-1}x\right),\quad\left(xy^{-1}yz^{-1}z\right)=\left(xz^{-1}zy^{-1}y\right)
\]
since idempotents in pregroupoids with unique preinverses commute,
so
\begin{gather}
\left[yy\left[zzx\right]\right]=\left[zz\left[yyx\right]\right],\quad\left[\left[xyy\right]zz\right]=\left[\left[xzz\right]yy\right].\label{eq:lh4}
\end{gather}
We also have $\left(xx^{-1}x\right)=x$, so
\begin{equation}
\left[xxx\right]=x.\label{eq:lh5}
\end{equation}
We call a non-empty set $P$ with a total function $\mathsf{t}$ satisfying
(\ref{eq:lh1}), (\ref{eq:lh4}) and (\ref{eq:lh5}) a \emph{preheap}
or, in Wagner's terminology, a \emph{generalized heap;} if $\mathsf{t}$
is a partial function satisfying (\ref{eq:lh1}), (\ref{eq:lh4})
and (\ref{eq:lh5}) when all terms are defined, we get a \emph{preheapoid}
instead.

\subsubsection*{Generalized semiheapoids}

We can let not only the binary operation on a semigroupoid $P$ but
also the involution $*$ on $P$ be a partial function. Then $\left[xyz\right]$
is defined if and only if $\left(x\left(y^{*}\right)z\right)$, and
then $\left[xyz\right]=\left(x\left(y^{*}\right)z\right)$. Such generalized
semiheapoids can be defined naturally for semigroupoids of matrices
(see Section \ref{subsec:Involution-magmoids}).

\end{document}